\documentclass[12pt]{amsart}

\usepackage[all,color]{xy}
\usepackage{pb-diagram}
\usepackage[mathscr]{eucal}
\usepackage{hyperref}
\hypersetup{
    colorlinks=true, 
    linktoc=all,     
    linkcolor=blue,  
}

\usepackage{xcolor}

\usepackage[toc,page]{appendix}
\usepackage{pifont}
\usepackage{combelow} 

\RequirePackage{ifthen,setspace,enumitem,tikz}
\usetikzlibrary{arrows}

\DeclareMathAlphabet{\mathpzc}{OT1}{pzc}{m}{it}

\usepackage{amsfonts}
\usepackage{amsmath}
\usepackage{amssymb}

\newcommand{\tr}{\textnormal{tr}}
\newcommand{\ric}{\textnormal{Ric}}

\newcommand{\RCD}{\textnormal{RCD}}

\DeclareMathOperator{\reg}{reg}

\def\cK{\mathcal{K}}

\def\PSH{\textnormal{PSH}}

\newtheorem{theorem}{Theorem}[section]
\newtheorem{proposition}{Proposition}[section]
\newtheorem{lemma}{Lemma}[section]
\newtheorem{definition}{Definition}[section]
\newtheorem{corollary}{Corollary}[section]

\numberwithin{equation}{section}

\newcommand{\ddc}{i\partial\bar\partial}
\newcommand{\eps}{\epsilon}
\newcommand{\vol}{\operatorname{Vol}}
\newcommand{\EntRic}{\operatorname{Ent}_{\mathrm{Ric}}}

\hypersetup{pdftitle={Ricci entropy, RCD structures and Kahler spaces},
 pdfauthor={Bin Guo, Jian Song and Jacob Sturm}}
\begin{document}
\raggedbottom
\title{Ricci entropy, RCD structures and K\"ahler spaces}
\author{Bin Guo $^*$, Jian Song $^\dagger$ and Jacob Sturm $^{\dagger\dagger}$}
\address{$^*$ Department of Mathematics \& Computer Science, Rutgers University, Newark, NJ 07102}
\email{bguo@rutgers.edu}
\address{$^\dagger$ Department of Mathematics, Rutgers University, Piscataway, NJ 08854}
\email{jiansong@math.rutgers.edu}
\address{$^{\dagger\dagger}$ Department of Mathematics \& Computer Science, Rutgers University, Newark, NJ 07102}
\email{sturm@newark.rutgers.edu}

\thanks{Research supported in part by the National Science Foundation under grants DMS-2203607 and DMS-2505575.}

\begin{abstract}
This paper is the final installment in our series on the geometric
theory of complex Monge-Amp\`ere equations. We study singular
K\"ahler metrics on compact normal K\"ahler spaces with klt
singularities whose volume densities lie in $L^p$ for some $p>1$. We show that a lower bound for the Ricci current in the sense of pluripotential theory is equivalent to a synthetic Ricci lower bound
in the sense of RCD theory. As a consequence, every K\"ahler class on
a compact normal K\"ahler space with klt singularities admits an RCD
structure, which makes differential geometric analysis available on
singular complex spaces. As an application, we show that the
fundamental group of a compact K\"ahler Calabi--Yau space with klt
singularities is almost Abelian. We further establish compactness
results for singular K\"ahler--Einstein spaces.
\end{abstract}
\maketitle

\section{Introduction}

The complex Monge--Amp\`ere equation is one of the fundamental links
between analysis and the geometry of K\"ahler manifolds. Its solutions
determine K\"ahler metrics, but estimates for the solutions do not
immediately give estimates for the associated metric spaces. This
distinction becomes particularly important when the volume measure
degenerates or when the underlying complex space is singular. In such
situations, the potential may remain bounded while the metric admits
no uniform upper bound, and even the finiteness of the diameter, or
the identification of the metric completion with the original complex
space, requires additional arguments.

The present paper is the final installment in our series devoted to
the geometric theory of complex Monge--Amp\`ere equations. Diameter
and Sobolev estimates, H\"older regularity and geometric stability
were developed in the preceding works
\cite{GPSS1,GPSS2,GPSS3,GKSSHolder,GSSStability}.
The main goal of the present paper is to establish the RCD structure
of singular K\"ahler metrics with Ricci curvature bounded below, in
arbitrary dimension. This provides a common framework in which
estimates from nonlinear PDE, comparison geometry on metric measure
spaces and holomorphic methods can be used together. We shall be
particularly interested in its consequences for fundamental groups,
projective embeddings and the stability of the induced distances.

The analytic starting point is Yau's solution of the Calabi conjecture
\cite{Yau}, together with the existence theory of Aubin and Yau for
K\"ahler--Einstein metrics with negative Ricci curvature
\cite{Aubin,Yau}. For a fixed compact K\"ahler manifold $(M,\theta)$,
we consider the complex Monge--Amp\`ere equation
\[
 (\theta+\ddc\varphi)^n=e^{-\psi} \theta^n, \qquad \int_M e^{-\psi} \theta^n=\int_M\theta^n,
 \qquad \sup_M\varphi=0.
\]
In his fundamental work \cite{Kolodziej98}, Ko\l odziej established a
uniform $L^\infty$ bound for $\varphi$ when $e^{-\psi} $ is bounded in $L^p$
for some $p>1$. His subsequent stability theorem \cite{Kolodziej03}
shows that, under a common $L^p$ bound, convergence of the densities
in $L^1$ implies uniform convergence of the normalized potentials.
He further proved the H\"older continuity of the potentials under the
same integrability assumption \cite{Kolodziej08}. These results
provide boundedness, regularity and continuous dependence of weak
solutions. Their extensions to degenerate equations and to singular
varieties, in particular \cite{EGZ,DP}, make it possible to construct
canonical currents on spaces arising in birational geometry.

Each of these three analytic properties has a natural geometric
counterpart. A diameter bound is an $L^\infty$ bound for the distance
function on $M\times M$; a H\"older estimate compares the intrinsic
distance with a fixed background distance; and geometric stability
asks whether nearby volume measures determine nearby distance
functions. All of these questions concern the metric
$\theta+\ddc\varphi$ itself, and hence involve two derivatives of a
potential for which only weak estimates may be available. The
analogy with Ko\l odziej's theory therefore serves as a guide for the
geometric theory, rather than as a formal consequence of the
corresponding estimates for the potentials.

An early step in this direction is the study of the Riemannian
geometry of K\"ahler--Einstein currents in \cite{SongCurrents}.
For projective Calabi--Yau varieties and canonical models with
crepant singularities, the metric completion of the regular locus is
identified there with the underlying variety, and its tangent cones
are studied. In \cite{SongBPF}, this approach leads to an analytic
proof of Kawamata's base point free theorem for smooth projective
manifolds with big and nef canonical bundle, combining the
K\"ahler--Ricci flow, Riemannian convergence and $L^2$ estimates
for holomorphic sections. This illustrates how the geometry of a
singular canonical metric can produce the holomorphic sections needed
to construct the canonical model.

Uniform gradient and diameter estimates for families of geometric
complex Monge--Amp\`ere equations, including collapsing families of
twisted K\"ahler--Einstein metrics, were established by Fu--Guo--Song
in \cite{FGSGeometric}. These estimates
extend the geometric study beyond a single canonical current.
Subsequently, Song--Sturm--Wang \cite{SSWCompactness} developed a
compactness theory for K\"ahler--Einstein manifolds with negative
Ricci curvature and volume bounded above. The limits are complete
metric spaces associated with semi-log canonical models, with the
non-klt locus removed. Combined with their work on the
Weil--Petersson potential \cite{SSWWeilPetersson}, this gives a
metric interpretation of important features of the compactification
of moduli spaces of general type. The appearance of complete
noncompact components is an essential feature of this theory when the
limit leaves the klt locus.

The breakthrough is the surprising removal of curvature assumptions from basic
geometric estimates. Guo--Phong--Song--Sturm established diameter and
Green's function
estimates in \cite{GPSS1}, as well as Sobolev inequalities, heat
kernel bounds and spectral estimates on singular K\"ahler spaces in
\cite{GPSS2}. These estimates rely on integral control of the volume
density together with the nondegeneracy hypotheses stated there, and
do not require any bound on the Ricci curvature. The small degeneracy
assumption was later removed in \cite{GPSS3}. In this way, the
diameter estimate provides a geometric counterpart of the $L^\infty$
estimate, while the Sobolev and heat kernel estimates provide the
analytic tools needed to study the singular metric itself.

Local curvature information provides further geometric control.
Guo--Song \cite{GSLocal} obtained local volume noncollapsing from a
Ricci lower bound on the ball under consideration. More recently,
their work on Nash entropy and Calabi energy \cite{GSCalabi}
connected estimates for the Laplace equation with RCD structures,
under suitable hypotheses on a resolution. The works of
Sz\'ekelyhidi \cite{Szekelyhidi, Szekelyhidi2} and Fu--Guo--Song \cite{FGS}
develop this connection for singular K\"ahler--Einstein metrics and
for more general singular K\"ahler metrics, and the latter
establishes the RCD structure in complex dimension three. The
positivity results of Chen--Chiu--Hallgren--Sz\'ekelyhidi--T\^o--Tong
\cite{CCHSTT} provide another important local ingredient. A common
difficulty in these works is to pass from curvature information on
the smooth locus to a global statement on its metric completion.

One of the analytic contributions of this paper is a natural
improvement of the RCD approximation criterion of Guo--Song
\cite[Proposition 3.1 and Corollary 3.2]{GSCalabi}. Their argument
requires a uniform $L^2$ bound for the negative part of the Ricci
curvature of the smooth approximating metrics. In
Proposition~\ref{prop:analytic}, we replace this curvature hypothesis
by a uniform $L\log L$ bound, formulated as a bound on the Ricci
entropy, within the approximation framework of the present paper.
This condition is weaker even than the $L^{2-\varepsilon}$
extension noted in \cite[Remarks 2.1 and 3.1]{GSCalabi}.
The key observation is that the K\"ahler structure gives an estimate
for the logarithmic gradients of eigenfunctions; a capacity argument,
together with the Ricci lower bound on the target, then yields their
Lipschitz regularity. This allows the RCD argument to proceed without
a uniform quadratic curvature bound along the regularization. We note
that the Ricci entropy used here measures negative Ricci curvature,
and is different from the Nash entropy of the volume density used in
our previous works.

The H\"older and stability parts of the theory complete this
picture. In \cite{GKSSHolder}, Guo--Ko\l odziej--Song--Sturm used
partial $C^0$ estimates and Bergman regularization to obtain uniform
H\"older estimates on polarized manifolds, and proved the H\"older
continuity of K\"ahler--Einstein potentials on smoothable projective
varieties. This connects the intrinsic geometry with the geometry of
projective embeddings. The corresponding geometric stability on a
fixed compact K\"ahler manifold was established by Guo--Song--Sturm
\cite{GSSStability}: under common $L^p$ and Ricci lower bounds, $L^1$
convergence of the volume densities implies uniform convergence of
the distance functions. We extend this result to klt spaces, using
the RCD theorem below and
the uniform distance estimate of \cite{GSSHolder}.

The curvature-dimension theory of Lott--Villani and Sturm
\cite{LottVillani,SturmCD}, together with its Riemannian formulation
by Ambrosio--Gigli--Savar\'e and Erbar--Kuwada--Sturm
\cite{AGS,EKS}, allows a lower bound for the Ricci curvature to be
expressed in terms of a distance and a measure.
Theorem~\ref{thm:main} places the singular K\"ahler metrics
considered here within this theory. In this picture, nonlinear PDE
provides the metric and its analytic estimates, RCD theory provides
global comparison and compactness, and holomorphic sections then
recover the complex and algebraic structure. The applications below
make this interplay concrete: almost nonnegative Ricci curvature
constrains the fundamental group, partial $C^0$ estimates give
uniform projective embeddings, and geometric stability identifies the
distance determined by a weak solution of the complex
Monge--Amp\`ere equation.

We now state our results more precisely.

Let $X$ be a   compact normal K\"ahler space of complex dimension
$n\ge2$ with klt singularities, and write $X_{\reg}$ for its
regular locus. We fix a smooth K\"ahler form $\theta$ on $X$, where
smoothness is understood in the sense of local ambient embeddings, and
a smooth adapted volume measure $\Omega$. By the latter we mean that a
smooth Hermitian metric on the $\mathbb Q$-line bundle $K_X$ has been
fixed. On a log resolution $\pi:Y\to X$ which is an isomorphism over
$ X_{\reg}$, we have
\begin{equation}\label{eq:adapted}
 K_Y=\pi^*K_X+\sum_{i=1}^m a_iE_i,\qquad a_i>-1,
 \qquad \pi^*\Omega=\prod_i|s_i|_{h_i}^{2a_i}\Omega_Y.
\end{equation}
where $\Omega_Y$ is a smooth positive volume form,
$s_i$ is a defining section of $E_i$, and $h_i$ is a smooth Hermitian
metric on $\mathcal O_Y(E_i)$.
The measure $\Omega$ is smooth and positive on $ X_{\reg}$ and has finite
mass on $X$, and its Ricci form $\ric(\Omega)$ is a smooth form on $X$.

\begin{definition}\label{def:class}
For $p>1$ and $K>0$, we define $\cK_{\theta,\Omega}(p,K)$ to be the set
of closed positive currents $\omega\in[\theta]$ with bounded local
potentials whose Monge--Amp\`ere measure satisfies
\begin{equation}\label{eq:class}
 \omega^n=e^{-\psi}\Omega,\qquad
 \|e^{-\psi}\|_{L^p(X,\Omega)}\le K.
\end{equation}
\end{definition}

If $\omega$ is smooth and strictly positive on $X_{\reg}$, we define the metric measure space $(\overline{X}, d, \mu)$ to be the metric completion of $( X_{\reg},\omega, \omega^n)$, i.e.,
$$(\overline{X}, d, \mu) = \overline{(X_{\reg}, \omega, \omega^n)}$$
where $\mu$ is the trivial extension of the volume measure  measure $\omega^n$. By the diameter estimate of \cite{GPSS1}, $(\overline{X}, d, \mu)$ is compact. 

The following is the main theorem of this paper.

\begin{theorem} \label{thm:main}
Let $X$ be an $n$-dimensional   compact normal K\"ahler space with klt singularities, equipped with a smooth
K\"ahler form $\theta$ and a smooth adapted volume measure $\Omega$.
Suppose that $\omega\in\cK_{\theta,\Omega}(p,K)$, for some $p>1$ and
$K>0$, is a smooth K\"ahler metric on $ X_{\reg}$ satisfying
\[
 \ric(\omega)\ge-\omega\qquad\text{on } X_{\reg}.
\]
Then $(\overline X,d,\mu)$ is a compact non-collapsed
$\RCD(-1,2n)$ space. 
\end{theorem}

As noted in \cite{FGS}, for any K\"ahler class on $X$, one can construct infinitely many singular K\"ahler metrics satisfying the assumptions in Theorem \ref{thm:main}. The Ricci curvature $\ric(\omega)$ on $X_{\reg}$ uniquely extends to a global current that is bounded below by $-\omega$ in the sense of pluripotential theory. Theorem \ref{thm:main} essentially establishes the equivalence between a lower bound of Ricci current and the synthetic Ricci lower bound in RCD theory.  One can further construct singular K\"ahler metric with Ricci current bounded both below and above. The following corollary follows directly from Theorem \ref{thm:main} combined with results in \cite{Szekelyhidi, FGS, Szekelyhidi2}.  

\begin{corollary}  \label{cor:main}
Under the hypotheses of Theorem~\ref{thm:main}, $(\overline X,d)$ is
homeomorphic to $X$, and its metric regular locus satisfies
$\mathcal R(\overline X)= X_{\reg}$.
In particular, the metric singular set
$\mathcal S(\overline X)$ satisfies
\begin{equation}\label{eq:cor-codim-three}
 \dim_{\mathcal H,d}\mathcal S(\overline X)\le2n-3.
\end{equation}
If, in addition, $\ric(\omega)\le \omega$ on $ X_{\reg}$, then
\begin{equation}\label{eq:cor-codim-four}
 \dim_{\mathcal H,d}\mathcal S(\overline X)\le2n-4.
\end{equation}
\end{corollary}

In Corollary \ref{cor:main}, the identity map on $ X_{\reg}$ in $\overline{X}$ to itself extends to a homeomorphism $\overline X\simeq X$. We emphasize that Corollary~\ref{cor:main} requires neither
rationality nor projectivity. Its proof, given in
Section~\ref{subsec:local-singularities}, relies on the K\"ahler
current estimate of \cite[Theorem 1.2]{CCHSTT}, local weighted
holomorphic sections, and the holomorphic charts of
\cite{LiuSzekelyhidi} for metrics with Ricci curvature bounded below.
These local ingredients allow us to adapt the density gap and cone
exclusion arguments of \cite{Szekelyhidi,FGS}.

 Immediately, we have the following corollary for singular K\"ahler-Einstein spaces. 

\begin{corollary}\label{cor:ke}
Let $X$ be an $n$-dimensional   compact normal K\"ahler space with klt singularities. Let $\omega\in[\theta]$
be a closed positive current with bounded local potentials whose
restriction to $ X_{\reg}$ is a smooth K\"ahler metric satisfying
\[
 \ric(\omega)=\lambda\omega\qquad\text{on } X_{\reg},
 \qquad \lambda\in\mathbb R.
\]
Then its metric completion $
 (\overline X,d, \mu )$ 
is a compact noncollapsed $\RCD(\lambda,2n)$ space hommeomorphic to $X$. 
\end{corollary}

Corollary~\ref{cor:ke} identifies the topological and holomorphic
structures of the K\"ahler--Einstein completion. The normal complex
structure of $X$ extends the given complex structure on the metric
regular locus, and the canonical identification is biholomorphic
with respect to this structure. In particular, the corollary applies
to K\"ahler--Einstein currents on klt log Fano varieties with ample
$-K_X$, and to the canonical K\"ahler--Einstein currents on canonical
models of general type, with Einstein constants normalized to $1$
and $-1$, respectively; see \cite{EGZ,SongCurrents,Szekelyhidi}.
It extends the three-dimensional K\"ahler--Einstein conclusion of
\cite[Corollary 1.1]{FGS} to arbitrary dimension.

More generally, we can assign canonical distances to currents that
need not be smooth on $ X_{\reg}$. Using the uniform H\"older estimate and the geometric stability results of Guo--Song--Sturm \cite{GKSSHolder, GSSHolder, GSSStability} extend to compact
normal klt K\"ahler spaces.  

We keep the background data $(X,\theta,\Omega)$
fixed, set $V=\int_X\theta^n$, and let $d_\theta$ be the
background length distance on $X$. 
For $p>1$, $K>0$ and $\Lambda\ge0$, we consider the family
\[
 \mathscr M(p,K,\Lambda)
 =\{\eta\in\cK_{\theta,\Omega}(p,K):
       \eta|_{ X_{\reg}}\text{ is smooth K\"ahler},\quad
       \ric(\eta)\ge-\Lambda\eta\text{ on } X_{\reg}\}.
\]
Theorem~\ref{thm:main} and Corollary~\ref{cor:main}  show that each member of this family has a
compact noncollapsed $\RCD(-\Lambda,2n)$ completion.

For $\lambda>0$, we define the following class of possibly nonsmooth
currents:
\begin{equation}\label{eq:rough-stability-class}
 \cK_{\theta,\Omega}(p,K;\lambda)
 =\left\{\omega\in\cK_{\theta,\Omega}(p,K):
 f_\omega:=-\log\frac{\omega^n}{\Omega}
       \in\operatorname{PSH}(X,\lambda\theta)\right\}.
\end{equation}
The condition on $f_\omega$ means that the logarithmic density
admits a $\lambda\theta$-plurisubharmonic representative.
In particular, the currents in this class are not required to be
smooth on $ X_{\reg}$. Then the stability theorem [Theorem 1.2] \cite{GSSStability} immediately extends to the following theorem.

\begin{theorem} 
\label{thm:canonical-distance}
For every $\omega\in\cK_{\theta,\Omega}(p,K;\lambda)$
there exists a canonical distance $d_\omega$ on $X$ such that $ (X,d_\omega,\omega^n)$ is a compact noncollapsed $\RCD(-\Lambda,2n)$-space
for some $\Lambda$ that depends only on
$X,\theta,\Omega,p,K,\lambda$. Furthermore, the identity map
from $(X,d_\theta)$ to $(X,d_\omega)$ is a Lipschitz homeomorphism whose inverse is Holder continuous. 
\end{theorem}

By canonical distance, we mean that it is uniquely determined and does not depend on regularization $\omega$ that mains a uniform Ricci lower bound (c.f. \cite{GSSStability}). Therefore every $\omega\in\cK_{\theta,\Omega}(p,K;\lambda)$ induces a unique compact K\"ahler RCD space homeomorphic to $X$ itself. 

The proof of Theorem~\ref{thm:main} proceeds in two steps. We first
treat, in Proposition~\ref{prop:smoothtwist}, the case of a smooth
twist with a Ricci lower bound. In Section~\ref{sec:proof}, we then
adapt the regularized maximum construction of
\cite[Section 6 and Lemmas 2.4--2.5]{FGS} to compact K\"ahler
spaces, using a logarithmic pole along the singular locus and
a fixed integrable bound for the approximating densities.
Spectral convergence then recovers the stated Ricci lower bound
on the original metric completion.
 We emphasize that Corollary~\ref{cor:main} requires neither
rationality nor projectivity. Its proof, given in
Section~\ref{subsec:local-singularities}, relies on the K\"ahler
current estimate of \cite[Theorem 1.2]{CCHSTT}, local weighted
holomorphic sections, and the holomorphic charts of
\cite{LiuSzekelyhidi} for metrics with Ricci curvature bounded below.
These local ingredients allow us to adapt the density gap and cone
exclusion arguments of \cite{Szekelyhidi,FGS}.

\subsection{Applications to fundamental groups of klt Calabi--Yau spaces}
\label{subsec:fundamental-applications}

The topology of compact K\"ahler spaces with nef anticanonical
class lies at a natural meeting point of differential and algebraic
geometry. When $c_1(X)=0$ and $X$ is smooth, it follows from Yau's
theorem \cite{Yau} and the Beauville--Bogomolov decomposition
\cite{Beauville} that the fundamental group is almost Abelian.
Here a group is called \emph{almost Abelian} if it contains an
Abelian subgroup of finite index. Replacing the vanishing of $c_1(X)$
by the nefness of $-K_X$ leads to a much larger class of spaces, on
which a Ricci-flat metric is no longer available.

Demailly--Peternell--Schneider \cite{DPS} studied this class by means
of metrics with almost nonnegative Ricci curvature and the geometry
of the Albanese map. The work of P\u aun on the fundamental group
\cite{Paun97}, combined with his surjectivity theorem for the
Albanese map \cite{PaunAlbanese}, establishes almost-Abelianity
for compact K\"ahler manifolds with nef anticanonical bundle.
In the projective case, the decomposition theorem of
Cao--H\"oring \cite{CaoHoring} gives a further structural
description. However, for singular spaces and log canonical pairs,
neither a smooth metric nor such a decomposition theory is available
in general.

A geometric approach to this problem was introduced by
Fu--Guo--Song--Wang \cite{FGSW}. They construct singular metrics with
almost nonnegative Ricci curvature, and apply the RCD Margulis lemma of
Deng--Santos-Rodr\'iguez--Zamora--Zhao \cite{RCDMargulis}
to obtain virtual nilpotence. Their surjectivity theorem for the
Albanese map, together with Campana's argument for K\"ahler groups
\cite{CampanaNilpotent}, then yields almost-Abelianity.
The construction relies on control of the lengths of generators of
the fundamental group, and does not require a uniform diameter
bound for the whole degenerating family. This method separates
the metric input from the complex geometric argument, and applies
to log canonical boundaries as well as to the Calabi--Yau case.

Theorem~\ref{thm:main} provides the required RCD input in all
dimensions. Consequently, the arguments of \cite{FGSW} yield the
following unconditional version of their Theorem~1.1, which extends
their Corollary~1.1 beyond the smooth and three-dimensional cases.

\begin{samepage}
\begin{theorem} 
\label{thm:fundamental}
Let $X$ be a   compact normal K\"ahler space with klt
singularities, and let $\Delta$ be an effective $\mathbb R$-divisor
such that $(X,\Delta)$ is log canonical. If $-(K_X+\Delta)$ is nef,
then the topological fundamental group $\pi_1(X)$ is almost Abelian.
There is no restriction on the complex dimension of $X$.
\end{theorem}
\end{samepage}

\begin{samepage}
\begin{corollary}\label{cor:calabi-yau}
Let $X$ be a   compact normal klt K\"ahler space and let
$\Delta\ge0$ be an $\mathbb R$-divisor such that $(X,\Delta)$ is
log canonical and $c_1(K_X+\Delta)=0$ in real Bott--Chern cohomology.
Then $\pi_1(X)$ is almost Abelian. In particular, this holds for
every compact klt K\"ahler Calabi--Yau space, in arbitrary dimension.
\end{corollary}
\end{samepage}

Both the general theorem and its Calabi--Yau specialization are
obtained by combining the dimension-independent arguments of
\cite{FGSW} with the RCD and topological conclusions above.  We note that the fundamental group here is that of the underlying
space $X$ and  does not concern $\pi_1(X_{\reg})$. The almost Abelian property for $X_{\reg}$ is investigated in \cite{FGSW2}. 

\subsection{Partial \texorpdfstring{$C^0$}{C0} estimates for singular K\"ahler spaces}
\label{subsec:partialc0-application}

The partial $C^0$ estimate can be viewed as a quantitative version of
the Kodaira embedding theorem. This estimate was introduced by Tian in his study of
K\"ahler--Einstein metrics on Fano surfaces \cite{Tian90}.
Donaldson--Sun \cite{DonaldsonSun} combined the construction of
holomorphic peak sections with the theory of Ricci limit spaces to
establish uniform projective embeddings and the algebraicity of
Gromov--Hausdorff limits. The estimate plays a central role in the
work of Chen--Donaldson--Sun \cite{CDS1,CDS2,CDS3} and Tian
\cite{TianYTD} on the Yau--Tian--Donaldson conjecture for Fano manifolds. 

On singular K\"ahler spaces, these arguments require an intrinsic
geometric theory at the singular points. Fu--Guo--Song \cite[Theorem 1.3]{FGS} obtained a
uniform partial $C^0$-estimate in complex dimension 3.
Theorem~\ref{thm:partialc0} below generalizes this result to higher
dimensions. We do not impose any smoothability assumption
on $X$.

Before stating the result, we fix the normalization of the Bergman
kernel. Let $L$ be an ample line bundle on $X$, and let $h$ be a
Hermitian metric on it with bounded local weights, smooth on
$ X_{\reg}$, whose curvature is $\sqrt{-1}\Theta_h(L)=\omega$.
With our curvature convention, $[\omega]=2\pi c_1(L)$.
For $\ell\ge1$, we define the inner product and the Bergman kernel by
\begin{equation}\label{eq:bergman-innerproduct}
 \langle s,t\rangle_\ell
 =\int_{ X_{\reg}}\langle s,t\rangle_{h^\ell}
       \frac{(\ell\omega)^n}{n!}
 \quad\text{on }H^0(X,L^\ell),
 \qquad
 \rho_{\ell,\omega}(x)=\sum_{\alpha=1}^{N_\ell}|s_\alpha(x)|_{h^\ell}^2,
\end{equation}
where $(s_\alpha)$ is an orthonormal basis and
$N_\ell=\dim H^0(X,L^\ell)$. The pointwise norms on $X$ are
understood by continuous extension from $ X_{\reg}$, using
Corollary~\ref{cor:main} and the section estimates below.

\begin{theorem} \label{thm:partialc0}
Fix $n\ge2$, $D>0$ and $v>0$. There exist an integer
$m=m(n,D,v)\ge1$ and constants $b,B>0$, depending only on
$n,D,v$, with the following property.
Let $(X,\omega)$ satisfy the hypotheses of
Theorem~\ref{thm:main}, and suppose that $(L,h)$ is a
polarization as above and that
\[
 -\omega\le\ric(\omega)\le\omega\quad\text{on } X_{\reg},
 \qquad \operatorname{diam}(\overline X,d)\le D,
 \qquad \int_{ X_{\reg}}\frac{\omega^n}{n!}\ge v.
\]
Then, for every integer $k\ge1$ and every $x\in X$,
\begin{equation}\label{eq:partialc0-main}
 b\le\rho_{mk,\omega}(x)\le B.
\end{equation}
Moreover, $m$ can be chosen so that $L^m$ is very ample, and
the resulting embedding realizes $X$ as a subvariety of
$\mathbb P^N$ for an integer $N=N(n,D,v)$.
\end{theorem}

Theorem~\ref{thm:partialc0} provides one of the main analytic
ingredients for a geometric approach to the Yau--Tian--Donaldson
correspondence on singular Fano varieties: uniform embeddings allow a
metric limit along a continuity method to be studied as an algebraic
degeneration. This is precisely the mechanism used in the smooth and
conical settings \cite{CDS1,CDS2,CDS3,TianYTD,SzekelyhidiContinuity}.
It suggests a route to the log Fano correspondence which does not
rely on a variational framework, giving existence in the K-stable
case and degeneration to a K-polystable K\"ahler--Einstein pair in
the K-semistable case. More precisely, existence on the original pair
corresponds to K-polystability, and a strictly K-semistable pair need
not admit a K\"ahler--Einstein metric.

For comparison, we recall that the general correspondence has been
established in the variational and algebro-geometric works
\cite{BermanK,BBJ,LTWUniform,LiUniform,LXZ}.
A geometric proof based on the present estimate would also require
uniform bounds along the chosen continuity path and the
identification of its algebraic limit. For a nonzero log boundary, it
would further require a logarithmic partial $C^0$ estimate, or an
approximation for which the relevant constants remain uniform, since
the metric in Theorem~\ref{thm:partialc0} is smooth on all of
$ X_{\reg}$. We do not carry out these additional steps here.

The same estimate is also the analytic ingredient needed to compare
the differential geometric moduli of K\"ahler--Einstein metrics with
the algebraic moduli of K-stable Fano varieties and of canonical
models of general type.   This extends the
perspective of \cite{DonaldsonSun,OdakaModuli,LWXModuli} and, in the
singular three-dimensional setting, of \cite[Section 11]{FGS}.
Together with the existence and uniqueness of the canonical metrics,
their compatibility in families, and the local deformation and
quotient theory, it provides a route towards a holomorphic
equivalence of the corresponding moduli spaces.  In the case of general type, the compact klt setting considered here
includes canonical models with canonical singularities. However, the
full KSBA boundary also contains semi-log canonical varieties, and
requires the complete, possibly noncompact, metric limits studied by Song--Sturm--Wang \cite{SSWCompactness,SSWWeilPetersson}. 

\bigskip
\noindent\textbf{Acknowledgments.}
The use of the Ricci entropy was inspired by a joint project of the
second author with Wangjian Jian and Yalong Shi on applications of
Kato bounds in K\"ahler geometry. The authors would like to thank
Wangjian Jian and Yalong Shi for many inspiring and insightful
discussions. The proof of Proposition~\ref{prop:analytic}, which is
the key technical ingredient of this paper, was provided by
ChatGPT-6 Astra.

\bigskip

\section{Preliminaries}\label{sec:prelim}

\subsection{Geometric estimates without curvature assumptions}

We fix a K\"ahler form $\theta_Y$ on the smooth compact resolution $Y$
and set $\vartheta=\pi^*\theta$. If we let
$h=\pi^*\Omega/\theta_Y^n$, then the klt inequalities in
\eqref{eq:adapted} imply that $h\in L^{1+\delta}(Y,\theta_Y^n)$ for
some $\delta>0$.

\begin{lemma} \label{lem:density}
If $e^{-\psi}\in L^p(X,\Omega)$ for $p>1$, then there exists $r>1$
such that $e^{-\pi^*\psi}h\in L^r(Y,\theta_Y^n)$, with a bound
depending only on $p$, the $L^p$ bound and the fixed resolution data.
\end{lemma}
\begin{proof}
We choose $1<r<p$ close enough to one so that
$r(p-1)/(p-r)<1+\delta$. By H\"older's inequality, we have
\begin{align*}
 \int_Ye^{-r\pi^*\psi}h^r\theta_Y^n
 &=\int_Ye^{-r\pi^*\psi}h^{r-1}\,\pi^*\Omega\\
 &\le\left(\int_Xe^{-p\psi}\Omega\right)^{r/p}
 \left(\int_Yh^{r(p-1)/(p-r)}\theta_Y^n\right)^{(p-r)/p}.
\end{align*}
The second factor is finite by the choice of $r$.
\end{proof}

The same computation, combined with the elementary inequality
$t|\log t|^N\le C_{N,r}(1+t^r)$, controls all finite entropy moments
with respect to $\theta_Y^n$. After principalizing the Jacobian
ideal, we may write $\vartheta^n=J\theta_Y^n$, where $J$ is locally
comparable to a product of nonnegative powers of the
$|s_i|_{h_i}^2$. Hence $|\log J|$ has finite Lebesgue moments of all
orders, and H\"older's inequality gives
\begin{equation}\label{eq:nashentropy}
 \int_{ X_{\reg}}\left|\log\frac{\omega^n}{\theta^n}\right|^N
                  \omega^n<\infty\qquad(N<\infty).
\end{equation}
We remark that this verification does not use the RCD conclusion.

The geometric estimates used in this paper come from the work of
Guo--Phong--Song--Sturm on diameter estimates and Sobolev inequalities
\cite{GPSS1,GPSS2,GPSS3}. In particular, the heat kernel and Sobolev
space results used below are taken from \emph{Sobolev inequalities on
K\"ahler spaces} \cite{GPSS2}. By an \emph{approximation} we mean a
family of smooth K\"ahler metrics
\begin{equation}\label{eq:approximation}
 \omega_j=\vartheta+\eps_j\theta_Y+\ddc u_j,
 \quad \eps_j\downarrow0,\quad
 \sup_j\|u_j\|_\infty<\infty,\quad
 \sup_j\left\|\frac{\omega_j^n}{\theta_Y^n}\right\|_{L^r(\theta_Y^n)}<\infty
\end{equation}
for a fixed $r>1$, which converge locally smoothly to $\pi^*\omega$
away from $E=\bigcup_iE_i$. All constants in this definition are
required to be uniform in $j$.
These hypotheses place the family in the setting of
\cite[Theorems 2.1 and 2.2]{GPSS2}. Indeed, the volumes are bounded
above and below, and the intersection numbers with
$[\theta_Y]^{n-1}$ are bounded. The $L^r$ bound controls each fixed
Nash entropy. Moreover, local smooth convergence gives a common
positive lower bound for the normalized volume densities on each
compact subset of $Y\setminus E$. Using an exhaustion and a partition
of unity, we therefore obtain a single continuous function $\gamma$,
positive on $Y\setminus E$ and vanishing on $E$, which lies below all
these densities. Since $E$ is a divisor, we have
$\dim_{\rm H}\{\gamma=0\}\le2n-2<2n-1$, where the Hausdorff dimension
is computed with respect to the fixed smooth metric $\theta_Y$. Hence
the lower density hypothesis in \cite{GPSS2} is satisfied by a common
function $\gamma$; the corresponding constants may depend on this
fixed approximating family.

\begin{theorem} \label{thm:background}
For the target metric with bounded potentials and the approximating
families considered here, the metric completion is compact and the
Dirichlet form has compact resolvent. Moreover, there exist $q>1$ and
uniform constants such that
\begin{align}
 \operatorname{diam}(\overline X,d)&\le C,
       &\operatorname{diam}(Y,\omega_j)&\le C,\label{eq:diameter}\\
 \|w\|_{L^{2q}}^2&\le C_S\left(\|\nabla w\|_2^2+\|w\|_2^2\right),
       &&\label{eq:sobolev}\\
 0\le H_j(t,x,y)&\le C(1+t^{-q/(q-1)})\qquad(0<t\le1),
       &&\label{eq:heat}\\
 \lambda_{k,j}&\ge c k^{(q-1)/q}\qquad(k\ge1).
       &&\label{eq:spectrum}
\end{align}
The Sobolev inequality holds for the target metric as well as for the
approximating metrics, and the Friedrichs heat kernel of the target
satisfies the analogous bound. The exponent $q$ is subcritical; the
optimal Sobolev exponent is not needed here.
\end{theorem}

The Sobolev, spectral and heat kernel estimates apply to the
approximations defined above. For completeness, we note that the
exponent in \eqref{eq:heat} also follows from \eqref{eq:sobolev}:
interpolation between $L^1$ and $L^{2q}$ gives a Nash inequality of
effective dimension $2q/(q-1)$, and the standard energy inequality for
the heat equation then yields the stated ultracontractivity bound.
Combined with the semigroup identity, this bounds the heat kernel, and
integrating the kernel along the diagonal bounds the eigenvalue
counting function. We emphasize that no curvature bound for the
approximating metrics is used in these preliminary estimates. The
passage to the target metric and its compact completion is the
singular space part of \cite{GPSS2}; the compactness and the passage
to the target needed in the present approximation setting are also
explained below.

\subsection{Spectral theory}
Let $M= X_{\reg}$. Since the local potentials are bounded, there exist
cutoff functions $\chi_k\in C_c^\infty(M)$ such that
\begin{equation}\label{eq:cutoffs}
 \begin{gathered}
 0\le\chi_k\le1,\qquad
 \chi_k=1\text{ on an exhausting sequence of compact subsets},\\
 \int_M|\nabla\chi_k|^2dV_g\longrightarrow0.
 \end{gathered}
\end{equation}
This is proved in \cite[Lemma 8.2]{GPSS2}. Following
\cite[Definition 8.1 and Proposition 8.1]{GPSS2}, we denote by
$W^{1,2}(X)$ the Sobolev space of finite energy functions for the
metric $g$ on $M$: its elements are functions in $W^{1,2}_{loc}(M)$
with finite norm in \eqref{eq:domain}, modulo equality almost
everywhere. By means of the cutoff functions of zero capacity, this
space is identified with the form domain:
\begin{equation}\label{eq:domain}
 W^{1,2}(X)=\overline{C_c^\infty(M)}^{\|\cdot\|_{1,2}},\qquad
 \|w\|_{1,2}^2=\int_M(w^2+|\nabla w|^2)dV_g.
\end{equation}
The same space is also naturally identified with
$W^{1,2}(\overline X,d,dV_g)$. We note that the notation $W^{1,2}(X)$
refers to the given metric and measure, and does not presuppose a
topological identification of $X$ with $\overline X$.
The following lemma makes explicit the removal step used below.

\begin{lemma} \label{lem:removal}
Suppose that \eqref{eq:cutoffs} holds. If $u\in W^{1,2}_{loc}(M)$ has
finite $\|u\|_{1,2}$, then $u\in W^{1,2}(X)$. If in addition $u$ is
smooth on $M$, $F\in L^2(M)$, and $\Delta u\ge-F$ there, then
\begin{equation}\label{eq:weakremove}
 \int_M\langle\nabla u,\nabla\zeta\rangle dV_g
 \le\int_MF\zeta\,dV_g
\end{equation}
for every bounded nonnegative $\zeta\in W^{1,2}(X)$.
\end{lemma}
\begin{proof}
Let $T_Nu=\max(-N,\min(u,N))$. For fixed $N$, we have
\[
 \|\nabla(\chi_kT_Nu)-\nabla T_Nu\|_2
 \le\|(1-\chi_k)\nabla T_Nu\|_2+N\|\nabla\chi_k\|_2\longrightarrow0.
\]
The $L^2$ convergence follows from the dominated convergence theorem,
and local smoothing shows that $\chi_kT_Nu\in W^{1,2}(X)$. We then
let $N\to\infty$. To prove the weak inequality, we test the
inequality on $M$ against $\chi_k\zeta$, using a local approximation
which preserves nonnegativity and a uniform bound. The resulting
error term is bounded in absolute value by
$\|\zeta\|_\infty\|\nabla u\|_2\|\nabla\chi_k\|_2$, which tends to
zero, and the remaining terms converge in $L^2$.
\end{proof}

\begin{proposition}\label{thm:criterion}
If $\ric(g)\ge-g$
on $M$ and all eigenfunctions of $\Delta$ are Lipschitz,
then $\overline{( X_{\reg}, \omega, \omega^n)}$ is a noncollapsed
$\RCD(-1,2n)$-space.
\end{proposition}
This is Honda's criterion for almost smooth spaces \cite{Honda}, in
the form with ordinary capacity explained in
\cite[Corollary 8]{Szekelyhidi}; see also
\cite[Proposition 2.3]{CCHSTT} and \cite[Lemma 3.3]{FGS}.
The remaining hypotheses of the criterion are the Sobolev-to-Lipschitz
property, the compactness of the energy embedding, and the Ricci lower
bound on the regular locus. The first follows from the local
Sobolev-to-Lipschitz property of smooth metrics together with the fact
that the distance is the intrinsic length distance, and the second
follows from the analytic background above. We note that this
criterion does not identify the metric completion with the algebraic
or analytic space $X$.

\subsection{Holomorphic equivalence}
\begin{theorem} \label{thm:szek}
Let $X$ be normal and projective, and let $\omega$ have bounded local
potentials and belong to $c_1(L)$ for an ample line bundle $L$. Assume
that $\omega$ is K\"ahler--Einstein on $ X_{\reg}$, that
$\omega\ge c\theta_{FS}$, that its completion is noncollapsed RCD
with curvature parameter equal to the Einstein constant,
and that $\omega^n/\theta_{FS}^n\in L^{1+\delta}(\theta_{FS}^n)$ for
some $\delta>0$. Then its metric completion is homeomorphic to $X$.
\end{theorem}
This is the statement of \cite[Theorem 17]{Szekelyhidi} for polarized
K\"ahler--Einstein metrics. We note that the identification of the
metric regular set is an additional geometric conclusion, and is not a
formal consequence of the homeomorphism. We shall not apply this
theorem directly to the metrics in Theorem~\ref{thm:main}. Instead,
Section~\ref{subsec:local-singularities} establishes the local version
needed here: bounded local potentials replace the global polarization,
\cite[Theorem 1.2]{CCHSTT} provides the K\"ahler current bound, and
the holomorphic charts of \cite{LiuSzekelyhidi} replace the use of
smooth convergence of K\"ahler--Einstein metrics.

\section{Sobolev to Lipschitz}\label{sec:bootstrap}

Throughout Sections~\ref{sec:bootstrap}, \ref{sec:riccientropy}
and~\ref{sec:proof}, the Laplacian is the real Laplacian
$\Delta=\operatorname{div}\nabla$, and the eigenvalue equation is
written as $-\Delta f=\lambda f$. In particular, $\Delta=2\Delta_\omega$,
where $\Delta_\omega f=\tr_\omega(\ddc f)$. All tensor norms and
volume measures in these sections are taken with respect to the
Riemannian metric.

\subsection{A K\"ahler refinement of the differential Kato inequality}
For a smooth real function $f$, we write $H=\nabla^2f$ for its real
Hessian and decompose
\begin{equation}\label{eq:hesssplit}
 H=B+B',\qquad
 B(V,W)=\frac{H(V,W)+H(JV,JW)}2.
\end{equation}
Here $B$ and $B'$ are the $J$-invariant and the $J$-anti-invariant
parts of the real Hessian, respectively. These two parts are
orthogonal to each other.

\begin{lemma}\label{lem:katoalgebra}
The following inequality holds almost everywhere:
\begin{equation}\label{eq:refinedkato}
 |H|^2\ge 2\bigl|\nabla|\nabla f|\bigr|^2
                -4|B|\bigl|\nabla|\nabla f|\bigr|.
\end{equation}
\end{lemma}
\begin{proof}
We identify a symmetric tensor with the corresponding self-adjoint
endomorphism. Under this identification, $B'J=-JB'$. If $e$ is a unit
vector, then $e,Je$ are orthonormal, and
\[
 |B'|^2\ge |B'e|^2+|B'Je|^2=2|B'e|^2.
\]
At a point where $|\nabla f|>0$, we let $e=\nabla f/|\nabla f|$.
For every vector $V$, we have
\[
 d|\nabla f|(V)=\frac{\langle\nabla_V\nabla f,\nabla f\rangle}{|\nabla f|}
       =H(V,e)=\langle He,V\rangle,
\]
and therefore $He=\nabla|\nabla f|$. Consequently,
\begin{align*}
 |H|^2=|B'|^2+|B|^2
 &\ge2|He-Be|^2+|B|^2\\
 &\ge2\bigl|\nabla|\nabla f|\bigr|^2
          -4|B|\bigl|\nabla|\nabla f|\bigr|.
\end{align*}
Finally, the weak gradient of the locally Lipschitz function
$|\nabla f|$ vanishes almost everywhere on its zero set. This proves
that the inequality holds almost everywhere.
\end{proof}

\subsection{The logarithmic differential inequality}
We set
\begin{equation}\label{eq:sh}
 h=\log(1+|\nabla f|^2).
\end{equation}
The functions $h$ and $e^{h/2}$ are smooth, even at points where
$\nabla f=0$. A direct computation gives
\begin{equation}\label{eq:gradhexact}
 \nabla h=\frac{2H(\nabla f)}{1+|\nabla f|^2},\qquad
 |\nabla h|^2
 =\frac{4|H(\nabla f)|^2}{(1+|\nabla f|^2)^2}
 \le\frac{4|H|^2}{1+|\nabla f|^2}.
\end{equation}
The last bound alone does not yield the required global energy
estimate, since at this stage no global $L^2$ bound for the full
Hessian is available.

\begin{lemma} \label{lem:drift}
Suppose that $-\Delta f=\lambda f$ and
$\ric(g)\ge-\rho g$ on a smooth K\"ahler manifold, where
$\lambda\ge0$ and $\rho\ge0$. Then, with the notation above,
\begin{equation}\label{eq:drift}
 \Delta h\ge-C(1+\rho)
       -Ce^{-h/2}|B||\nabla h|-Ce^{-h}|B|^2.
\end{equation}
Here $C$ depends on an upper bound for $\lambda$. Moreover, the
inequality also holds at the critical points of $f$.
\end{lemma}
\begin{proof}
Recall Bochner's identity
\[
 \frac12\Delta|\nabla f|^2
 =|H|^2+\ric(\nabla f,\nabla f)-\lambda|\nabla f|^2.
\]
On $\{|\nabla f|>0\}$, we use the identity
\[
 \bigl|\nabla(|\nabla f|^2)\bigr|^2
 =4|\nabla f|^2\bigl|\nabla|\nabla f|\bigr|^2.
\]
Combining Lemma~\ref{lem:katoalgebra} with $e^h-|\nabla f|^2=1$, we
obtain
\begin{equation}\label{eq:rawdrift}
 \begin{split}
 \Delta h
 &=e^{-h}\Delta|\nabla f|^2
       -e^{-2h}\bigl|\nabla(|\nabla f|^2)\bigr|^2\\
 &\ge4e^{-2h}\bigl|\nabla|\nabla f|\bigr|^2
      -8e^{-h}|B|\bigl|\nabla|\nabla f|\bigr|-2(\lambda+\rho).
 \end{split}
\end{equation}
On $\{|\nabla f|\ge1\}$, the negative mixed term satisfies
\[
 8e^{-h}|B|\bigl|\nabla|\nabla f|\bigr|
 =\frac{4|B|}{|\nabla f|}|\nabla h|
 \le4\sqrt2 e^{-h/2}|B||\nabla h|.
\]
On $\{0<|\nabla f|<1\}$, completing the square gives
\begin{align*}
 4e^{-2h}\bigl|\nabla|\nabla f|\bigr|^2
       -8e^{-h}|B|\bigl|\nabla|\nabla f|\bigr|
 &\ge-4|B|^2\\
 &\ge-8e^{-h}|B|^2.
\end{align*}
These two estimates prove \eqref{eq:drift} away from the critical
set. At points where $\nabla f=0$, a direct differentiation together
with Bochner's identity gives $\Delta h=2|H|^2\ge0$, and hence the
same inequality holds there.
\end{proof}

\subsection{Lipschitz property from Sobolev estimates}
\begin{lemma} 
\label{lem:bootstrap}
Suppose $\omega$ satisfies the assumptions in Theorem \ref{thm:main}. Let $M=X_{\reg}$ and let 
$f\in W^{1,2}(X)$ be the eigenfunction satisfying  $-\Delta f=\lambda f$. If
\begin{equation}\label{eq:seedtarget}
 \int_M(h^2+|\nabla h|^2)\,dV_g<\infty,
 \qquad \int_M|B|^2\,dV_g<\infty,
\end{equation}
then $e^{h/2}\in W^{1,2}(X)\cap L^\infty(X)$. In particular,
$f$ extends to a Lipschitz function on the metric completion.
\end{lemma}
\begin{proof}
We point out that the Ricci lower bound in the statement is a tensor
inequality on the smooth locus $M$; we do not assume any tensorial
Ricci curvature bound on the singular completion. First,
Lemma~\ref{lem:removal} shows that $h$ lies in $W^{1,2}(X)$.
Since $e^h\in L^1$ and $|B|\in L^2$, Lemma~\ref{lem:drift} gives
\begin{equation}\label{eq:targetdrift}
 \Delta h\ge-C-Ce^{-h/2}|B||\nabla h|-Ce^{-h}|B|^2
 \quad\hbox{on }M,
\end{equation}
where the right hand side is integrable.

For $R>0$, we define the bounded increasing function
\begin{equation}\label{eq:nonlineartest}
 \Phi_R(h)=\frac{e^h}{1+e^h/R}.
\end{equation}
We record its derivative and the cancellation needed below:
\begin{equation}\label{eq:testidentities}
 \begin{gathered}
 \Phi_R'=\frac{e^h}{(1+e^h/R)^2},\qquad
 \frac{\Phi_R^2}{e^h\Phi_R'}=1,\\
 0\le\Phi_R\le\min(e^h,R),\qquad e^{-h}\Phi_R\le1.
 \end{gathered}
\end{equation}
For fixed $R$, we test \eqref{eq:targetdrift} against the compactly
supported smooth nonnegative function $\chi_k\Phi_R(h)$. Integrating
by parts, with the sign convention for $\Delta$, we obtain
\begin{align}
 \int_M\chi_k\Phi_R'|\nabla h|^2
 &\le C\int_M\chi_k\Phi_R
 +C\int_M\chi_k\Phi_Re^{-h/2}|B||\nabla h|
 +C\int_M\chi_k\Phi_Re^{-h}|B|^2\notag\\
 &\quad-\int_M\Phi_R\langle\nabla\chi_k,\nabla h\rangle.
 \label{eq:bootstrapibp}
\end{align}
Here and in the rest of this proof, all integrals are taken with
respect to $dV_g$. The last term has no definite sign, but its
absolute value is bounded by $R\|\nabla\chi_k\|_2\|\nabla h\|_2$.
By Young's inequality and the middle identity in
\eqref{eq:testidentities}, we have
\[
 C\chi_k\Phi_Re^{-h/2}|B||\nabla h|
 \le\frac12\chi_k\Phi_R'|\nabla h|^2+C\chi_k|B|^2.
\]
It follows that
\begin{equation}\label{eq:fixedR}
 \frac12\int_M\chi_k\Phi_R'|\nabla h|^2
 \le C\int_M(e^h+|B|^2)+R\|\nabla\chi_k\|_2\|\nabla h\|_2.
\end{equation}
Since $\Phi_R'\le R/4$, the integrand on the left is integrable for
fixed $R$. We first let $k\to\infty$, using the vanishing of the
ordinary $2$-capacity, and only then let $R\to\infty$. Since the
functions $\Phi_R'(h)$ increase to $e^h$, the monotone convergence
theorem gives
\begin{equation}\label{eq:gradientenergy}
 \int_M e^h|\nabla h|^2\le C\int_M(e^h+|B|^2)<\infty.
\end{equation}
Now $|\nabla(e^{h/2})|^2=e^h|\nabla h|^2/4$ and
$\int_M|e^{h/2}|^2=\int_M e^h<\infty$, so
Lemma~\ref{lem:removal} proves $e^{h/2}\in W^{1,2}(X)$.
This is the crucial gain in the argument: neither a weighted cutoff
estimate nor a $W^{1,2+\delta}$ capacity statement has been used.

Next, the Bochner identity and the inequality
$|H(\nabla f)|^2\le |H|^2|\nabla f|^2$ yield
\begin{align}
 \Delta(e^{h/2})
 &=e^{-h/2}\bigl(|H|^2+\ric(\nabla f,\nabla f)
                          -\lambda|\nabla f|^2\bigr)\notag\\
 &\quad-e^{-3h/2}|H(\nabla f)|^2\notag\\
 &\ge -(\lambda+1)e^{h/2}.\label{eq:zbochner}
\end{align}
Indeed, the Hessian terms contribute at least $e^{-3h/2}|H|^2\ge0$,
and $e^{-h/2}|\nabla f|^2\le e^{h/2}$.
Since $e^{h/2}\in W^{1,2}(X)$ and
$(\lambda+1)e^{h/2}\in L^2$, Lemma~\ref{lem:removal} extends
\eqref{eq:zbochner} to all bounded nonnegative test functions in
$W^{1,2}(X)$.

For completeness, we include the standard Moser iteration. Let
$w_L=\min(e^{h/2},L)$, with $L\ge1$, and use $w_L^{m-1}$ as a test
function, where $m\ge2$. At each stage where $e^{h/2}\in L^m$, this
gives
\[
 \frac{4(m-1)}{m^2}\int_M|\nabla w_L^{m/2}|^2
 \le (\lambda+1)\int_M e^{h/2}w_L^{m-1}\le (\lambda+1)\int_M e^{mh/2}.
\]
Applying \eqref{eq:sobolev} and then letting $L\to\infty$, we obtain
\[
 \|e^{h/2}\|_{mq}\le
 \left[C_S\left(1+\frac{(\lambda+1)m^2}{4(m-1)}\right)\right]^{1/m}
 \|e^{h/2}\|_m.
\]
Starting from $m=2$ and iterating with $m=2q^j$, we see that the
product of the constants converges, since $q>1$. Hence
\begin{equation}\label{eq:mosergradient}
 \|e^{h/2}\|_\infty\le C(\lambda,C_S,q)\|e^{h/2}\|_2.
\end{equation}
In particular, the gradient of $f$ is bounded on $M$. By integrating
along piecewise smooth curves, we see that $f$ is Lipschitz with
respect to the intrinsic length distance, and hence it extends
uniquely to a Lipschitz function on $\overline X$.
\end{proof}

\section{Monge--Amp\`ere approximations}\label{sec:approximation}

The main goal of Section \ref{sec:approximation}, Section \ref{sec:riccientropy} and Section \ref{sec:bounding}   is to develop technical preparations to prove the following proposition for twisted K\"ahler-Einstein metrics. 

\begin{proposition}[The smooth-twist case]\label{prop:smoothtwist}
Let $X$ be a   compact normal K\"ahler space of complex
dimension $n\ge2$ with klt singularities. Let $\theta$ be a smooth
K\"ahler form and $\Omega$ a smooth adapted volume measure. Suppose
that $\omega\in\cK_{\theta,\Omega}(p,K)$ for some $p>1$ is smooth and
K\"ahler on $ X_{\reg}$ and satisfies
\[
 \ric(\omega)=-\omega+\beta\qquad\text{on } X_{\reg},
\]
where $\beta$ is a smooth non-negative real closed $(1,1)$-form with smooth
local potentials on all of $X$.  
Then the metric completion of $(X, \omega, \omega^n)$ 
is a compact noncollapsed $\RCD(-1,2n)$-space.
\end{proposition}

Proposition \ref{prop:smoothtwist} is a special case of Theorem \ref{thm:main}. In Sections~\ref{sec:approximation}--\ref{sec:bounding}, we establish
the estimates needed for the proof of this proposition, which is
given in Section~\ref{subsec:smoothtwistproof}. The positivity of
$\beta$ is not needed for the entropy estimates on the resolution.
If $\beta\ge0$, one can take $B=1$. We note that this intermediate
result does not require projectivity.

\subsection{Barrier functions}

Throughout Sections~\ref{sec:approximation}--\ref{sec:bounding},
the approximations on the resolution are constructed under the
hypotheses of Proposition~\ref{prop:smoothtwist}. In particular,
\begin{equation}\label{tke}
 \ric(\omega)=-\omega+\beta\qquad\text{on } X_{\reg},
\end{equation}
where $\beta$ is a fixed smooth real closed $(1,1)$-form with
smooth local potentials on $X$. The Ricci lower bound for the target
metric is the separate hypothesis $\ric(\omega)\ge-B\omega$.
We write $\omega=\theta+\ddc\varphi$, where
$\varphi\in\PSH(X,\theta)\cap L^\infty(X)$, and keep the notation
$\psi$ for the density potential in \eqref{eq:class}. On $ X_{\reg}$,
we set
\[
 F=\varphi+\psi.
\]
By the twisted equation, we have
\[
 \ddc F=\beta-\theta-\ric(\Omega)\qquad\text{on } X_{\reg}.
\]
We now verify that $F$ extends smoothly to $X$. Locally, the smooth
form $\beta-\theta-\ric(\Omega)$ admits a smooth real potential $v$,
and thus $F-v$ is pluriharmonic on the regular locus of that
neighborhood. A pluriharmonic function on the regular locus of a
normal complex space extends uniquely to a pluriharmonic function
across its singular set \cite[Corollary 2]{Vajaitu}. Such a function
is locally the real part of a holomorphic function, and hence is
smooth in the sense of ambient embeddings. Adding $v$ gives a smooth
local extension of $F$, and these local extensions agree by
uniqueness on the dense regular locus. Consequently, $F\in C^\infty(X)$
and
\begin{equation}\label{o-ma}
 (\theta+\ddc\varphi)^n=e^{\varphi-F}\Omega,
 \qquad \ddc F+\ric(\Omega)=\beta-\theta.
\end{equation}
The identity of measures holds on all of $X$, since both measures
give zero mass to $X\setminus  X_{\reg}$. We emphasize that it is $F$
which is smooth, whereas $\psi=F-\varphi$ need not extend smoothly.

We choose a log resolution $\pi:Y\to X$, which is an isomorphism over
$ X_{\reg}$ and also principalizes the Jacobian ideal. More precisely,
we choose a resolution obtained by blowups which admits an effective
exceptional divisor whose negative is relatively ample. On a compact
K\"ahler space, this gives a K\"ahler resolution together with smooth
metrics on the divisors such that, after taking a small positive
multiple of the exceptional coefficients,
\begin{equation}\label{eq:barrier}
 \vartheta=\pi^*\theta,\qquad
 \chi=\sum_i b_i\log |s_i|_{h_i}^2-C_0\le0,\quad b_i>0,
 \qquad \vartheta+\ddc\chi\ge c\theta_Y\quad\hbox{off }E.
\end{equation}
Here $E=\bigcup_iE_i$ and $\theta_Y$ is a fixed smooth K\"ahler form
on $Y$. This is the barrier construction along the exceptional
divisor used in \cite[proof of Proposition 3.1]{CCHSTT}; in
particular, it is available in the compact normal K\"ahler setting.
We note that the resolution is chosen to have this property, and we
do not claim it for an arbitrary modification. Geometrically, the
curvature of the relatively ample negative exceptional divisor
controls the contracted directions, while $\vartheta$ controls the
remaining directions. Choosing the coefficients small then gives a
positive form on $Y$.

Let $s_i$ be the defining section of $\mathcal O_Y(E_i)$ and let
$h_i$ be a smooth Hermitian metric on it, whose curvature $\Theta_i$
is normalized by
\begin{equation}\label{eq:poincarelelong}
 \ddc\log |s_i|_{h_i}^2=2\pi[E_i]-\Theta_i.
\end{equation}
After rescaling the $h_i$ by positive constants, we may assume that
$|s_i|_{h_i}^2\le e^{-2}$, and we absorb these changes into the smooth
volume form $\Omega_Y$ in \eqref{eq:adapted}. This changes neither the
curvature forms nor the discrepancies, and we keep the notation
$a_i>-1$ for the latter.

\subsection{Regularization of the volume measure}
Let $\phi=\pi^*\varphi$. By the twisted equation and the normality
argument above, $\pi^*F$ is a fixed smooth function on $Y$. In
particular,
\begin{equation}\label{eq:liftedpsi}
 \begin{gathered}
 \pi^*F=\phi+\pi^*\psi\quad\text{on }Y\setminus E,\\
 \ddc(\pi^*F)=\pi^*(\beta-\theta-\ric(\Omega))\ge-C\theta_Y
 \quad\text{on }Y.
 \end{gathered}
\end{equation}
The equation for the target metric becomes
\begin{equation}\label{eq:targetMA}
 (\vartheta+\ddc\phi)^n
 =e^{\phi-\pi^*F}\prod_i|s_i|_{h_i}^{2a_i}\Omega_Y.
\end{equation}
We keep $\pi^*F$ unchanged, and use the same parameter $\eps$ both to
perturb the class and to regularize the divisor factors. We define
\begin{equation}\label{eq:densities}
 \begin{split}
 \mathcal F_\eps
 &=e^{-\pi^*F}\prod_i(|s_i|_{h_i}^2+\eps)^{a_i}\frac{\Omega_Y}{\theta_Y^n},
       \qquad 0<\eps\le\tfrac14,\\
 \mathcal F_0
 &=e^{-\pi^*F}\prod_i|s_i|_{h_i}^{2a_i}\frac{\Omega_Y}{\theta_Y^n}.
 \end{split}
\end{equation}

\begin{lemma} \label{lem:regularizeddensity}
There exist $r>1$ and constants $c,C>0$, independent of $\eps$, such
that
\begin{equation}\label{eq:densityuniform}
 \begin{gathered}
 \|\mathcal F_\eps\|_{L^r(\theta_Y^n)}\le C,
 \qquad c\le\int_Y\mathcal F_\eps\theta_Y^n\le C,\\
 \mathcal F_\eps\longrightarrow\mathcal F_0
                   \quad\text{in }L^r(\theta_Y^n)\text{ as }\eps\downarrow0.
 \end{gathered}
\end{equation}
\end{lemma}
\begin{proof}
We choose $r>1$ close enough to one so that $ra_i>-1$ for every
negative discrepancy $a_i$. The smooth positive factor
$e^{-\pi^*F}\Omega_Y/\theta_Y^n$ is bounded above and below on $Y$.
For $a_i\ge0$, the factors $(|s_i|_{h_i}^2+\eps)^{a_i}$ are uniformly
bounded, while for $a_i<0$ we have
$(|s_i|_{h_i}^2+\eps)^{ra_i}\le |s_i|_{h_i}^{2ra_i}$ away from $E_i$.
In simple normal crossing coordinates, the resulting product
$\prod_{a_i<0}|s_i|_{h_i}^{2ra_i}$ is integrable. This proves the
uniform $L^r$ bound and, by the dominated convergence theorem, the
stated $L^r$ convergence. The upper bound for the mass follows from
H\"older's inequality. On a fixed ball compactly contained in
$Y\setminus E$, all factors have a common positive lower bound for
$0<\eps\le1/4$, and integrating over this ball gives the lower bound
for the mass.
\end{proof}

\subsection{Regularization by complex Monge--Amp\`ere equations}
For $0<\eps\le1/4$, we let
\begin{equation}\label{eq:alpha}
 \alpha_\eps=\vartheta+\eps\theta_Y.
\end{equation}
This is the smooth K\"ahler reference form in the approximating
class, and should not be confused with the solution metric. For each
$0<\eps\le1/4$, we solve
\begin{equation}\label{eq:approxMA}
 \begin{split}
 \omega_\eps&=\alpha_\eps+\ddc u_\eps>0,\\
 \omega_\eps^n
 &=e^{u_\eps-\pi^*F}
       \prod_i(|s_i|_{h_i}^2+\eps)^{a_i}\Omega_Y
   =e^{u_\eps}\mathcal F_\eps\theta_Y^n.
 \end{split}
\end{equation}
Because of the positive exponential factor, the existence and
uniqueness of a smooth solution follow from the classical theorem of
Aubin and Yau for Monge--Amp\`ere equations with positive exponential
\cite{Aubin,Yau}. In particular, we do not assume the solvability of
any Fano K\"ahler--Einstein equation on $Y$. The additive constant of
$u_\eps$ is determined by the equation, rather than by a separate
normalization. Therefore the factor $e^{u_\eps}$ suffices once the
fixed target density has been incorporated into
$\mathcal F_\eps$ defined in \eqref{eq:densities}.
We note that, in general, the factor $e^{-\pi^*F}$ cannot be dropped
while keeping $\Omega_Y$ fixed: doing so would replace the target
equation \eqref{eq:targetMA} by
$(\vartheta+\ddc\phi)^n=e^\phi\prod_i|s_i|_{h_i}^{2a_i}\Omega_Y$,
which is a different equation unless the Ricci potential is constant,
up to the corresponding normalization.

\begin{lemma} \label{lem:C0}
There exist $r>1$ and $C$ independent of $0<\eps\le1/4$ such that
\begin{equation}\label{eq:uniformbounds}
 \|u_\eps\|_\infty\le C,\qquad
 \left\|\frac{\omega_\eps^n}{\theta_Y^n}\right\|_{L^r(\theta_Y^n)}\le C.
\end{equation}
\end{lemma}
This follows from the uniform $L^\infty$ estimate for degenerate
complex Monge--Amp\`ere equations due to Eyssidieux--Guedj--Zeriahi
\cite[Theorems 2.1 and 4.1]{EGZ} and Demailly--Pali
\cite[Theorem 2.2(A),(C) and Lemma 2.14]{DP},
whose hypotheses are satisfied uniformly: $\vartheta\le\alpha_\eps\le C\theta_Y$,
$\int_Y\alpha_\eps^n\ge\int_Y\vartheta^n>0$, and
Lemma~\ref{lem:regularizeddensity} controls the density and its mass.
The second bound then follows from \eqref{eq:approxMA}.
We also refer to
\cite[Lemma 7.1 and the proof of Theorem 3.3 in Section 12]{GPSS2}
for these estimates in the setting of resolution approximations.

\subsection{Estimates for approximating metrics}
We define
\begin{equation}\label{eq:Psi}
 \Psi_\eps=\sum_{a_i>0}a_i\log(|s_i|_{h_i}^2+\eps),\qquad
 \sup_Y\Psi_\eps\le C,\qquad \ddc\Psi_\eps\ge-C\theta_Y.
\end{equation}
Taking the Ricci curvature of \eqref{eq:approxMA}, we obtain
\begin{align}
 \ric(\omega_\eps)
 &=-\omega_\eps+\alpha_\eps+\ddc(\pi^*F)
       +\ric(\Omega_Y)-\sum_i a_i\ddc\log(|s_i|_{h_i}^2+\eps)\notag\\
 &\ge-\omega_\eps-\ddc\Psi_\eps-C\theta_Y.
 \label{eq:ricciPsi}
\end{align}
Here we have used the uniform bound
$\ddc\log(|s_i|_{h_i}^2+\eps)\ge-C\theta_Y$. Thus the negative
discrepancies contribute terms bounded from below, while the positive
discrepancies are kept in $\Psi_\eps$.

\begin{lemma} \label{lem:inverse}
There exist $C,N>0$, independent of $\eps$, such that
\begin{equation}\label{eq:inverse}
 U=\tr_{\omega_\eps}\theta_Y\le C\prod_i|s_i|_{h_i}^{-2N}
 \quad\hbox{on }Y\setminus E.
\end{equation}
\end{lemma}
By the trace estimate of \cite[Lemma 7.1]{GPSS2}, we have
$\tr_{\theta_Y}\omega_\eps\le C\prod_i|s_i|_{h_i}^{-2N_0}$.
Here we use the modification for equations with positive exponential
described in the proof of Theorem 3.3 in Section 12 of that paper.
Its hypotheses follow from \eqref{eq:barrier},
Lemma~\ref{lem:regularizeddensity} and the fixed smooth factor
$e^{-\pi^*F}\Omega_Y/\theta_Y^n$;
in particular, the derivatives of $\log\mathcal F_\eps$ satisfy
uniform polynomial bounds away from $E$. To derive
\eqref{eq:inverse} from the cited estimate, we use
\[
 \frac{\omega_\eps^n}{\theta_Y^n}\ge c\prod_{a_i>0}|s_i|_{h_i}^{2a_i},
 \qquad
 \tr_{\omega_\eps}\theta_Y
 \le\frac{n(\tr_{\theta_Y}\omega_\eps)^{n-1}}
              {\omega_\eps^n/\theta_Y^n},
\]
and enlarge $N$. The lower bound for the determinant follows from
\eqref{eq:approxMA} and Lemma~\ref{lem:C0}.

\begin{corollary} \label{cor:approximation}
As $\eps\downarrow0$, the metrics $\omega_\eps$ converge locally
smoothly to $\pi^*\omega$ away from $E$. Consequently, every sequence
$\eps_j\downarrow0$ gives an approximation of $\omega$.
\end{corollary}
This follows from the convergence argument of \cite[Proposition 7.1 and the
proof of Theorem 3.3 in Section 12]{GPSS2}, applied to
\eqref{eq:approxMA}. The bounded limiting solution coincides with
$\phi$ by the uniqueness of solutions to the equation with positive
exponential \cite[Theorem 4.1]{EGZ}, and the required uniform bounds
are given by \eqref{eq:uniformbounds}.

\section{Ricci entropy}\label{sec:riccientropy}

\begin{definition}
For a smooth metric $g$ on $ X_{\reg}$, we define at each point
$x\in  X_{\reg}$
\[
 \rho(x)=\max\{0,-\lambda_{\min}(\ric(g)(x))\}.
\]
Its Ricci entropy is defined by
\[
 \EntRic(g)=\int_{ X_{\reg}}\rho_g\log(1+\rho_g)\,dV_g.
\]

\end{definition}

This is the pointwise negative Ricci function considered by
Carron--Mondello--Tewodrose \cite{CMT}; in particular,
$\ric(g)\ge-\rho_g g$. The same definition applies to a smooth compact
approximating manifold $Y$, with the integral taken over $Y$. If
$R\ge0$ satisfies $\ric(g)\ge-Rg$, then $\rho_g\le R$ and
$\EntRic(g)\le\int R\log(1+R)\,dV_g$.

\begin{proposition} \label{prop:analytic}
Let $\omega\in\cK_{\theta,\Omega}(p,K)$ be smooth and K\"ahler on
$ X_{\reg}$, with $\ric(\omega)\ge-\omega$ there, and let
$\{\omega_\eps\}$ be an approximation in the sense of
\eqref{eq:approximation}.
For $0<\eps\le1/4$, let $g_\eps$ be the Riemannian metric of
$\omega_\eps$. Suppose that
\begin{equation}\label{eq:entropyhyp}
 \rho_\eps=\rho_{g_\eps},\qquad
 \sup_{0<\eps\le1/4}\EntRic(g_\eps)
 =\sup_{0<\eps\le1/4}\int_Y\rho_\eps\log(1+\rho_\eps)\,dV_{g_\eps}<\infty.
\end{equation}
Then every eigenfunction of the Laplacian of the target metric is
Lipschitz, and the metric completion of the target, equipped with its
Riemannian measure, is a noncollapsed $\RCD(-1,2n)$ space.
\end{proposition}

The condition \eqref{eq:entropyhyp} can be formulated as a norm
bound. We write $|\ric(g)_-|:=\rho_g$ for the magnitude of the most
negative Ricci eigenvalue, and consider the Luxemburg norm
\[
 \|v\|_{L\log L(Y,dV_g)}
 :=\inf\left\{a>0:\int_Y\Phi\left(\frac{|v|}{a}\right)dV_g\le1\right\},
 \qquad \Phi(t)=t\log(1+t).
\]
Since $\Phi(2t)\le4\Phi(t)$, the entropy bound and the norm bound
are equivalent in the following uniform sense:
\[
 \eqref{eq:entropyhyp}
 \quad\Longleftrightarrow\quad
 \sup_{0<\eps\le1/4}
 \bigl\||\ric(g_\eps)_-|\bigr\|_{L\log L(Y,dV_{g_\eps})}<\infty.
\]
By comparison, the curvature assumption in
\cite[(3.3) and Proposition 3.1]{GSCalabi} is
\[
 \sup_{0<\eps\le1/4}
 \bigl\||\ric(g_\eps)_-|\bigr\|_{L^2(Y,dV_{g_\eps})}<\infty.
\]

Proposition~\ref{prop:analytic} provides a natural
improvement of the curvature hypothesis in the approach of Guo--Song
to RCD structures. We note that the density and convergence
assumptions in \eqref{eq:approximation} remain in force, while the
uniform $W^{2,2}$ and $W^{1,4}$ estimates of
\cite[Theorem 1.1]{GSCalabi} are not needed here.

In the smooth-twist case, Corollary~\ref{cor:approximation} already
shows that the metrics constructed there satisfy
\eqref{eq:approximation}. The proposition itself only uses the
stated approximation properties and the analytic background of
Section~\ref{sec:prelim}. We emphasize that the approximating metrics
are not required to have a uniform pointwise Ricci lower bound; the
curvature bound in the conclusion is the bound $-1$ of the target
metric.

\subsection{An integral Hessian identity}
\begin{lemma}\label{lem:hessian}
On a closed K\"ahler manifold, the component $B$ in
\eqref{eq:hesssplit} satisfies
\begin{equation}\label{eq:hessianidentity}
 \int_Y|B|^2\,dV_g=\frac12\int_Y(\Delta f)^2\,dV_g.
\end{equation}
In particular, for a normalized eigenfunction $-\Delta f=\lambda f$,
we have $\int_Y|B|^2\,dV_g=\lambda^2/2$.
\end{lemma}
\begin{proof}
For $n\ge2$, since the K\"ahler form is closed, Stokes' theorem gives
\[
 \int_Y\ddc f\wedge\ddc f\wedge\omega^{n-2}=0.
\]
At a given point, we choose a unitary frame which diagonalizes the
$J$-invariant real tensor $B$, and write its eigenvalues, which occur
in pairs, as $\nu_1,\nu_1,\ldots,\nu_n,\nu_n$. The trace identity for
the square of the real $(1,1)$-form $\ddc f$ then implies
\[
 \int_Y\left(\sum_\alpha \nu_\alpha^2\right)dV_g
 =\int_Y\left(\sum_\alpha \nu_\alpha\right)^2dV_g.
\]
Since $|B|^2=2\sum \nu_\alpha^2$ and $\Delta f=2\sum \nu_\alpha$,
this is exactly \eqref{eq:hessianidentity}. In complex dimension one,
the identity holds pointwise. We note that no Ricci curvature term
appears in this identity, which only controls the $J$-invariant
component of the Hessian.
\end{proof}

\subsection{Logarithmic gradient energy}
\begin{lemma}\label{lem:logenergy}
Let $(Y,g)$ be a closed K\"ahler manifold, and suppose that
$-\Delta f=\lambda f$ and $\|f\|_2=1$. We write $\rho=\rho_g$. Then,
for $\lambda$ in a fixed bounded subset of $[0,\infty)$,
\begin{equation}\label{eq:logenergy}
 \int_Y(h^2+|\nabla h|^2)\,dV_g
 \le C\left(\vol(Y,g)+\lambda+\lambda^2+
              \int_Y\rho\log(1+\rho)\,dV_g\right).
\end{equation}
\end{lemma}
\begin{proof}
We apply Lemma~\ref{lem:drift} and the chain rule to $(1+h)^2$:
\[
 \Delta(1+h)^2=2(1+h)\Delta h+2|\nabla h|^2.
\]
By Young's inequality, the term $C(1+h)e^{-h/2}|B||\nabla h|$ can be
absorbed into one copy of $|\nabla h|^2$, at the cost of
$C(1+h)^2e^{-h}|B|^2$. Since
$\sup_{h\ge0}(1+h)^2e^{-h}<\infty$, it follows that
\begin{equation}\label{eq:chainentropy}
 \Delta(1+h)^2
 \ge |\nabla h|^2-C|B|^2-C(1+h)(1+\rho).
\end{equation}
The elementary entropy inequality
\begin{equation}\label{eq:entropydual}
 \rho h\le \rho\log(1+\rho)+e^h\qquad(h\ge0,\ \rho\ge0)
\end{equation}
follows from $\rho h\le \rho\log\rho-\rho+e^h$, with the convention
$0\log0=0$.
Moreover, $\rho\le C(1+\rho\log(1+\rho))$ and
$h^2\le Ce^h$. Integrating \eqref{eq:chainentropy} over the closed
manifold and using
\[
 \int_Y e^h\,dV_g=\vol(Y,g)+\lambda,
 \qquad \int_Y|B|^2\,dV_g=\lambda^2/2,
\]
we obtain \eqref{eq:logenergy}. We note that this integration takes
place entirely on the smooth approximating manifold, and uses neither
integration by parts on the singular space nor any RCD estimate for
the target.
\end{proof}

\subsection{Spectral convergence}
\begin{lemma} \label{lem:spectral}
Let $0=\lambda_0\le\lambda_1\le\cdots$ be the eigenvalues of the
nonnegative Friedrichs Laplacian of the target metric, counted with
multiplicities. Then, for an approximation, $\lambda_{k,j}\to\lambda_k$
for every $k\ge0$. Moreover, after passing to a single subsequence,
there exist orthonormal real eigenbases $\{f_{k,j}\}_{k\ge0}$ on $Y$
such that
\[
 -\Delta_{g_j}f_{k,j}=\lambda_{k,j}f_{k,j},\qquad
 \lambda_{k,j}\longrightarrow\lambda_k,\qquad
 f_{k,j}\longrightarrow f_k\quad\hbox{in }C^\infty_{\mathrm{loc}}(M),
\]
where $\{f_k\}$ is an orthonormal eigenbasis for the target, including
all multiplicities.
\end{lemma}
\begin{proof}
We use the uniform heat kernel estimate of
\cite[Theorem 2.2]{GPSS2}, which was verified for the approximating
family in Section~\ref{sec:prelim}, and the density of
$C_c^\infty(M)$ in $W^{1,2}(X)$ established in
\cite[Proposition 8.1]{GPSS2}. Here $M$ is identified with
$Y\setminus E$, and all inner products and energy forms below are
taken with respect to the corresponding Riemannian measures.

First, we will obtain upper bound for the eigenvalues. Fix $k$ and $\delta>0$. We approximate the first $k+1$ elements of an
orthonormal eigenbasis of the target in $W^{1,2}(X)$ by functions
$\phi_0,\ldots,\phi_k\in C_c^\infty(M)$. These functions can be chosen
so that their span has dimension $k+1$ and every nonzero element
$\phi$ of the span has Rayleigh quotient at most $\lambda_k+\delta$.
Indeed, under these approximations, the $L^2$ Gram matrix and the
energy matrix converge to those of the chosen eigenfunctions. We
extend the $\phi_i$ by zero to $Y$. Since $g_j$ converges locally
smoothly on the common compact support of these functions, both
matrices computed with respect to $g_j$ converge as well. The min--max
principle then gives
\[
 \limsup_{j\to\infty}\lambda_{k,j}\le\lambda_k+\delta.
\]
Letting $\delta\downarrow0$, we obtain the upper bound. In particular,
the first $k+1$ eigenvalues of the approximating metrics are
uniformly bounded.

Now we will obtain the heat kernel bound for eigenfunctions. For an $L^2$-normalized eigenfunction $f_{i,j}$, the symmetry of the
heat kernel and the semigroup identity give, for $0<t\le\tfrac12$,
\begin{align*}
 |f_{i,j}(x)|
 &=e^{t\lambda_{i,j}}
   \left|\int_YH_j(t,x,y)f_{i,j}(y)\,dV_{g_j}(y)\right|\\
 &\le e^{t\lambda_{i,j}}
       \left(\int_YH_j(t,x,y)^2\,dV_{g_j}(y)\right)^{1/2}\\
 &=e^{t\lambda_{i,j}}H_j(2t,x,x)^{1/2}.
\end{align*}
Taking $t=\tfrac12$ in \eqref{eq:heat}, we obtain
$\|f_{i,j}\|_\infty\le C_k$ for $0\le i\le k$, uniformly in $j$.
We note that this estimate does not use any Ricci bound.

By local elliptic estimates and the local smooth convergence of the
metrics, after passing to a subsequence, we have
$\lambda_{i,j}\to\ell_i$ and
$f_{i,j}\to f_i$ in $C^\infty_{loc}(M)$ for $0\le i\le k$.
The uniform $L^r$ bound for the densities gives, for every
neighborhood $U$ of $E$,
\begin{equation}\label{eq:uniformintegrability}
 \sup_j\vol_{g_j}(U)\le C\left(\int_U\theta_Y^n\right)^{1-1/r}.
\end{equation}
Consequently,
\[
 \sup_j\int_U|f_{i,j}f_{m,j}|\,dV_{g_j}
 \le C_k^2\sup_j\vol_{g_j}(U)\longrightarrow0
 \quad\hbox{as }U\downarrow E.
\]
By local convergence and Fatou's lemma, the same bound holds for the
limiting products. Therefore, integrating first outside $U$ and then
letting $U\downarrow E$, we obtain
$\int_M f_i f_m\,dV_g=\delta_{im}$. In particular, no normalized
eigenfunction loses mass along the exceptional locus.

Applying Fatou's lemma on a compact exhaustion, we have
\[
 \int_M|\nabla f_i|^2\,dV_g
 \le\liminf_j\int_Y|\nabla f_{i,j}|^2\,dV_{g_j}=\ell_i.
\]
Hence $f_i\in W^{1,2}(X)$ by Lemma~\ref{lem:removal}. Passing to the
limit in the equations tested against $\phi\in C_c^\infty(M)$, we
obtain
\[
 \int_M\langle\nabla f_i,\nabla\phi\rangle\,dV_g
 =\ell_i\int_M f_i\phi\,dV_g.
\]
Both sides are continuous with respect to the $W^{1,2}(X)$ norm, and
by density the identity extends to every $\phi\in W^{1,2}(X)$.
Therefore the limit $f_i$ is a Friedrichs eigenfunction with
eigenvalue $\ell_i$, and its energy equals $\ell_i$.
The span of $f_0,\ldots,f_k$ has dimension $k+1$ and maximal
Rayleigh quotient $\ell_k$, so a second application of the min--max
principle yields
\[
 \lambda_k\le\ell_k
 \le\limsup_j\lambda_{k,j}\le\lambda_k.
\]
Since this argument applies to a further subsequence of any
subsequence, we conclude that $\lambda_{k,j}\to\lambda_k$ for the
whole sequence.

Finally, we take a diagonal subsequence over $k$ and over a compact
exhaustion of $M$. The limits are orthonormal eigenfunctions with
the ordered eigenvalues $\lambda_k$. Since the resolvent is compact,
each eigenspace is finite dimensional. For each eigenvalue, the number
of recovered orthonormal eigenfunctions equals its multiplicity, and
hence they span the corresponding eigenspace. Therefore the recovered
family is a complete eigenbasis. We note that no basis needs to be
prescribed in advance in an eigenspace of higher multiplicity.
\end{proof}

\subsection{Proof of Proposition~\ref{prop:analytic}}
\begin{proof}
We choose any sequence $\eps_j\downarrow0$ and apply
Lemma~\ref{lem:spectral} to the approximating metrics
$\omega_{\eps_j}$. Let $f$ be an element of the recovered eigenbasis
of the target. Its smooth approximations have bounded eigenvalues and
unit $L^2$ norm. By Lemmas~\ref{lem:hessian} and \ref{lem:logenergy},
the corresponding tensors $B_j$ and the logarithmic gradients satisfy
uniform $L^2$ bounds. By the local smooth convergence of the metrics
and the functions, together with Fatou's lemma on an exhaustion of
$M$, we obtain
\[
 \int_M|B|^2\,dV_g<\infty,
 \qquad \int_M(h^2+|\nabla h|^2)\,dV_g<\infty.
\]
It then follows from Lemma~\ref{lem:bootstrap} that $f$ is Lipschitz.
Since each eigenspace is finite dimensional and spanned by recovered
eigenfunctions, every eigenfunction is Lipschitz. The asserted
noncollapsed RCD conclusion now follows from
Proposition~\ref{thm:criterion}.
\end{proof}

\section{Bounding Ricci entropy}\label{sec:bounding}

Throughout this section, all estimates are uniform for
$0<\eps\le1/4$.

\subsection{Negative Ricci representation}
We define the smooth nonnegative $(1,1)$-forms
\begin{equation}\label{eq:spike}
 P_{i,\eps}=\ddc\log(|s_i|_{h_i}^2+\eps)
                   +\frac{|s_i|_{h_i}^2}{|s_i|_{h_i}^2+\eps}\Theta_i.
\end{equation}
If $Ds_i$ denotes the $(1,0)$ Chern derivative of the section $s_i$,
then a direct calculation in a local frame gives
\begin{equation}\label{eq:spikeformula}
 P_{i,\eps}
 =\frac{\eps}{(|s_i|_{h_i}^2+\eps)^2}
       i\langle Ds_i\wedge\overline{Ds_i}\rangle_{h_i}\ge0,
 \qquad P_{i,\eps}\le\frac{C}{|s_i|_{h_i}^2}\theta_Y\quad\hbox{off }E_i.
\end{equation}
For the last estimate, we use that the section and its covariant
derivative satisfy fixed smooth bounds, and that
$\eps/(|s_i|_{h_i}^2+\eps)^2\le1/(4|s_i|_{h_i}^2)$.

Substituting \eqref{eq:spike} into the Ricci identity
\eqref{eq:ricciPsi} and using the fixed lower bound for
$\ddc(\pi^*F)$, we obtain
\begin{equation}\label{eq:riccispike}
 \ric(\omega_\eps)\ge-\omega_\eps-C\theta_Y
                                  -\sum_{a_i>0}a_iP_{i,\eps}.
\end{equation}
Hence the smooth nonnegative function
\begin{equation}\label{eq:rho}
 R_\eps
 =1+C\tr_{\omega_\eps}\theta_Y+
       \sum_{a_i>0}a_i\tr_{\omega_\eps}P_{i,\eps}
\end{equation}
dominates the intrinsic negative Ricci function
$\rho_\eps:=\rho_{g_\eps}$, that is,
\[
 \ric(g_\eps)\ge-R_\eps g_\eps,
 \qquad 0\le\rho_\eps\le R_\eps.
\]
Here the constant term $1$ accounts for $-\omega_\eps$ in
\eqref{eq:riccispike}.

\begin{lemma} \label{lem:logrho}
There exists a constant $C$, independent of $\eps$, such that on
$Y\setminus E$,
\begin{equation}\label{eq:logrho}
 \log(1+R_\eps)\le C\left(1+\sum_k\bigl(-\log|s_k|_{h_k}^2\bigr)\right).
\end{equation}
In particular, the estimate holds almost everywhere on $Y$ for every
smooth approximating metric.
\end{lemma}
\begin{proof}
By \eqref{eq:inverse} and \eqref{eq:spikeformula}, there exist fixed
constants $C,N>0$ such that
\[
 R_\eps\le
 C\prod_i|s_i|_{h_i}^{-2N}\left(1+\sum_i|s_i|_{h_i}^{-2}\right)
 \quad\hbox{off }E.
\]
Taking logarithms and using $-\log|s_k|_{h_k}^2\ge2$, we obtain the
estimate. The last assertion follows since $E$ has zero volume.
\end{proof}

\subsection{The entropy estimate}

\begin{lemma}\label{lem:logmoment}
For each exceptional component $E_k$,
\begin{equation}\label{eq:logmoment}
 \int_Y\bigl(-\log|s_k|_{h_k}^2\bigr)\,\theta_Y\wedge\omega_\eps^{n-1}\le C.
\end{equation}
\end{lemma}
\begin{proof}
We expand $\omega_\eps^{n-1}-\alpha_\eps^{n-1}$ using
$\omega_\eps-\alpha_\eps=\ddc u_\eps$. By the Poincar\'e--Lelong
formula and integration by parts, we have
\begin{align*}
 &\left|\int_Y\bigl(-\log|s_k|_{h_k}^2\bigr)\theta_Y\wedge
              (\omega_\eps^{n-1}-\alpha_\eps^{n-1})\right|\\
 &\quad\le C\|u_\eps\|_\infty
 \left(\int_Y\theta_Y^2\wedge\alpha_\eps^{n-2}
       +\int_{E_k}\theta_Y\wedge\alpha_\eps^{n-2}\right)\le C.
\end{align*}
Here the mixed masses are bounded using $|\Theta_k|\le C\theta_Y$ and
$[\omega_\eps]=[\alpha_\eps]$, and all test forms are smooth for fixed
$\eps$. The reference integral is uniformly bounded since
$\alpha_\eps\le C\theta_Y$ and
$-\log|s_k|_{h_k}^2\in L^1(Y,\theta_Y^n)$.
\end{proof}

\begin{lemma} \label{lem:spikemoment}
For all $i,k$,
\begin{equation}\label{eq:spikemoment}
 \int_Y\bigl(-\log|s_k|_{h_k}^2\bigr)P_{i,\eps}\wedge\omega_\eps^{n-1}
 \le C+2\pi|\log\eps|\int_{E_k}\alpha_\eps^{n-1}.
\end{equation}
\end{lemma}
\begin{proof}
We apply the Poincar\'e--Lelong formula to \eqref{eq:spike}. After
integrating by parts, the terms involving $\Theta_i$ and $\Theta_k$
are uniformly bounded by Lemma~\ref{lem:logmoment}, since
$|\log(|s_i|_{h_i}^2+\eps)|\le-\log|s_i|_{h_i}^2$ and
$0\le |s_i|_{h_i}^2/(|s_i|_{h_i}^2+\eps)\le1$. Hence
\begin{align}
 &\int_Y\bigl(-\log|s_k|_{h_k}^2\bigr)P_{i,\eps}\wedge\omega_\eps^{n-1}\notag\\
 &\quad\le C-2\pi\int_{E_k}\log(|s_i|_{h_i}^2+\eps)
                                      \omega_\eps^{n-1}\notag\\
 &\quad\le C+2\pi|\log\eps|\int_{E_k}\alpha_\eps^{n-1}.
 \label{eq:spikeibp}
\end{align}
where the last step uses $\log(|s_i|_{h_i}^2+\eps)\ge\log\eps$ and
$[\omega_\eps]=[\alpha_\eps]$. This calculation remains valid at the
crossings of the divisors, since $\log(|s_i|_{h_i}^2+\eps)$ is smooth
for fixed $\eps>0$.
\end{proof}

\begin{lemma}\label{lem:exceptionalmass}
Let $Z_k=\pi(E_k)$, $d_k=\dim_{\mathbb C}Z_k$, and $c_k=n-d_k$.
Then $c_k\ge2$ and
\begin{equation}\label{eq:exceptionalmass}
 \begin{split}
 \int_{E_k}\alpha_\eps^{n-1}
 &=\sum_{r=0}^{d_k}\binom{n-1}{r}\eps^{n-1-r}
                   \int_{E_k}\vartheta^r\wedge\theta_Y^{n-1-r}\\
 &\le C_k\eps^{c_k-1}\le C_k\eps.
 \end{split}
\end{equation}
\end{lemma}
\begin{proof}
Since $Z_k\subset X_{\mathrm{sing}}$ and $X$ is normal, we have
$c_k\ge2$. The map $\pi|_{E_k}$ has rank at most $d_k$, and hence
$(\vartheta|_{E_k})^{d_k+1}=0$. Expanding
$(\vartheta+\eps\theta_Y)^{n-1}$ gives \eqref{eq:exceptionalmass},
since the coefficients of the expansion are fixed and nonnegative,
and the smallest power of $\eps$ that appears is
$n-1-d_k=c_k-1\ge1$.
\end{proof}

This is precisely where complex codimension at least two is used in
the entropy construction. For an isolated singularity, the mass is
exactly $\eps^{n-1}\int_{E_k}\theta_Y^{n-1}$, and for a center of
positive dimension it still tends to zero at least linearly. By
contrast, a nonexceptional divisor has $c_k=1$, and its mass in
general has a nonzero limit. For this reason, the calculation does
not give the same bound for cone singularities along divisors with
negative curvature concentrated there, such as cone angles greater
than $2\pi$.

\begin{proposition} \label{prop:entropy}
The solutions of \eqref{eq:approxMA} satisfy
\begin{equation}\label{eq:entropyestimate}
 \int_Y\rho_\eps\log(1+\rho_\eps)
                  \omega_\eps^{n}
 \le C\left(1+|\log\eps|
                      \sum_k\eps^{c_k-1}\right)\le C.
\end{equation}
Consequently, every sequence $\eps_j\downarrow0$ gives an
approximation $\omega_{\eps_j}$ with uniformly bounded
Ricci entropy.
\end{proposition}
\begin{proof}
Since the function $t\mapsto t\log(1+t)$ is increasing for $t\ge0$,
it suffices to bound the same integral with $R_\eps$ in
place of $\rho_\eps$. We multiply \eqref{eq:rho} by the upper bound
\eqref{eq:logrho}, and use the trace identity
\[
 (\tr_{\omega_\eps}\eta)\omega_\eps^n
       =n\eta\wedge\omega_\eps^{n-1}.
\]
The contribution of the constant term $1$ is uniformly bounded.
Indeed, if $r'=r/(r-1)$, then the uniform density bound and
H\"older's inequality give
\[
 \int_Y\left(1+\sum_k\bigl(-\log|s_k|_{h_k}^2\bigr)\right)\omega_\eps^n
 \le
 \left\|\frac{\omega_\eps^n}{\theta_Y^n}\right\|_{L^r(\theta_Y^n)}
 \left\|1+\sum_k\bigl(-\log|s_k|_{h_k}^2\bigr)\right\|_{L^{r'}(\theta_Y^n)}
 \le C.
\]
Here each function $-\log|s_k|_{h_k}^2$ has finite Lebesgue moments
of all orders, since $s_k$ is a fixed section defining a divisor,
measured with a smooth Hermitian metric.
The logarithmically weighted $\theta_Y$ terms are bounded by
Lemma~\ref{lem:logmoment}, and the logarithmically weighted
$P_{i,\eps}$ terms by Lemma~\ref{lem:spikemoment}. The corresponding
unweighted masses are also uniformly bounded; for the spike terms,
this follows from
\[
 \int_YP_{i,\eps}\wedge\omega_\eps^{n-1}
 =\int_Y\frac{|s_i|_{h_i}^2}{|s_i|_{h_i}^2+\eps}
                  \Theta_i\wedge\omega_\eps^{n-1}
 \le C\int_Y\theta_Y\wedge\alpha_\eps^{n-1};
\]
where we have used that the integral of
$\ddc\log(|s_i|_{h_i}^2+\eps)\wedge\omega_\eps^{n-1}$ is zero.
We have thus bounded the entropy of $R_\eps$, and since
$\rho_\eps\le R_\eps$, we obtain the more precise estimate
\begin{equation}\label{eq:entropyintersection}
 \int_Y\rho_\eps\log(1+\rho_\eps)
                       \omega_\eps^n
 \le C\left(1+|\log\eps|
                        \sum_k\int_{E_k}\alpha_\eps^{n-1}\right).
\end{equation}
Combined with Lemma~\ref{lem:exceptionalmass}, this gives
\eqref{eq:entropyestimate}.

Since $c_k-1\ge1$ and $0<\eps\le1/4$, we have
\[
 |\log\eps|\sum_k\eps^{c_k-1}
 \le C\eps|\log\eps|\longrightarrow0
 \qquad\hbox{as }\eps\downarrow0.
\]
This proves the uniform bound in \eqref{eq:entropyestimate}.
The assertion for every sequence $\eps_j\downarrow0$ then follows
from Corollary~\ref{cor:approximation}.
\end{proof}

\section{Proof of Theorem 1.1}\label{sec:proof}

\subsection{The smooth-twist case}\label{subsec:smoothtwistproof}

\begin{proof}[Proof of Proposition~\ref{prop:smoothtwist}]
By the twisted equation \eqref{tke} and the normality argument in
Section~\ref{sec:approximation}, the function
$F=\varphi+\psi$ appearing in \eqref{o-ma} is smooth.
For each $0<\eps\le1/4$, we solve \eqref{eq:approxMA} with the fixed
smooth factor $e^{-\pi^*F}$, and Lemma~\ref{lem:regularizeddensity}
gives uniform bounds for the density on its right hand side.
By Lemma~\ref{lem:C0}, Lemma~\ref{lem:inverse} and
Corollary~\ref{cor:approximation}, any sequence
$\omega_{\eps_j}$ with $\eps_j\downarrow0$ is an
approximation of the prescribed target metric, with a fixed density
exponent $r>1$, and all constants in the approximation are uniform.

We write $\rho_\eps=\rho_{g_\eps}$ for the intrinsic negative Ricci
functions.  \eqref{eq:rho} satisfies
$\rho_\eps\le R_\eps$, and by Proposition~\ref{prop:entropy} we have
\[
 \sup_{0<\eps\le1/4}\int_Y\rho_\eps\log(1+\rho_\eps)\,dV_{g_\eps}<\infty.
\]
Passing from $\omega_\eps^n$ to $dV_{g_\eps}=\omega_\eps^n/n!$
only changes a fixed constant factor. We note that all the estimates
above use the smoothness of the fixed form $\beta$ only through the
smooth function $F$, and do not require any sign condition on $\beta$.

We now apply Proposition~\ref{prop:analytic} after scaling by $B$.
The metrics $B\omega_\eps$ give a tame approximation of $B\omega$,
with background forms $B\theta$ and $B\theta_Y$, and the scaled target
satisfies $\ric(B\omega)=\ric(\omega)\ge-B\omega$.
For the corresponding Riemannian metrics, we have
\[
 \rho_{Bg_\eps}=B^{-1}\rho_{g_\eps},\qquad
 dV_{Bg_\eps}=B^n dV_{g_\eps}.
\]
Since $B\ge1$, it follows that
\[
 \EntRic(Bg_\eps)
 =B^{n-1}\int_Y\rho_{g_\eps}
             \log(1+B^{-1}\rho_{g_\eps})\,dV_{g_\eps}
 \le B^{n-1}\EntRic(g_\eps).
\]
Hence the required entropy bound holds uniformly for the scaled
approximation, and Proposition~\ref{prop:analytic} shows that
$B\omega$ has a noncollapsed $\RCD(-1,2n)$ completion. Scaling the
metric and its Riemannian measure back, we obtain the noncollapsed
$\RCD(-B,2n)$ assertion for $\omega$. Finally, rescaling the measure
by a constant gives the same RCD assertion with $\mu=\omega^n$.
\end{proof}

\subsection{Regularization on a compact K\"ahler space}
\label{subsec:kahlerreg}

We shall use the following versions, for compact normal K\"ahler spaces, of
the regularization and convergence arguments in
\cite[Lemmas 2.4--2.5 and 6.1]{FGS}.

\begin{lemma}
\label{lem:twistregularization}
Let $X$ be a compact normal K\"ahler space with klt singularities,
with smooth K\"ahler form $\theta$ and adapted measure $\Omega$.
Suppose that
\[
 \phi \in\PSH(X,A\theta)\cap C^\infty( X_{\reg}),\qquad
 \sup_X \phi = 0,\qquad e^{-\phi}\in L^p(X,\Omega),\quad p>1,
\]
where $A>0$ is fixed. Then there exist $1<p_0<p$, a  
function $e^{-\phi'}\in L^{p_0}(X,\Omega)$, a constant $A'>0$ and functions
$\phi_j\in C^\infty(X)$ such that
\begin{equation}\label{eq:qregularization}
 \begin{gathered}
 \phi_j\le0,\qquad \ddc \phi_j\ge-A'\theta,\qquad
 \phi_j\longrightarrow \phi\quad\text{in }C^\infty_{\mathrm{loc}}( X_{\reg}),\\
 1\le e^{-\phi_j}\le e^{-\phi'},\qquad
 e^{-\phi_j}\longrightarrow e^{-\phi}\quad\text{in }L^{p_0}(X,\Omega).
 \end{gathered}
\end{equation}
We do not claim that the sequence is monotone, nor that $\phi_j\ge \phi$.
\end{lemma}
\begin{proof}
We adapt the regularized maximum construction of
\cite[Lemma 2.4]{FGS}, using the logarithmic pole along
$X_{\mathrm{sing}}$ provided by \cite[Lemma 5]{DemaillyQ}.
By the klt integrability condition, a sufficiently small perturbation
by this pole admits a fixed bound in $L^{p_0}(X,\Omega)$, and a
smooth truncation then gives \eqref{eq:qregularization}.
This replaces the projective approximation from above used in
\cite{FGS}, and requires neither projectivity nor rationality.
\end{proof}

\begin{lemma}[Approximation by smooth twists]\label{lem:kahlerapprox}
Under the hypotheses of Theorem~\ref{thm:main}, there exist
constants $A\ge1$ and $p_0>1$, a nonnegative function
$H\in L^{p_0}(X,\Omega)$, smooth nonnegative  real closed $(1,1)$-forms
$\beta_j$  on $X$, and metrics $\omega_j\in[\theta]$ with
bounded potentials, smooth and K\"ahler on $ X_{\reg}$, such that
\begin{equation}\label{eq:kahlerapprox}
 \begin{gathered}
 \ric(\omega_j)=-A\omega_j+\beta_j\ge-A\omega_j,\\
 \omega_j\longrightarrow\omega
       \quad\text{in }C^\infty_{\mathrm{loc}}( X_{\reg}).
 \end{gathered}
\end{equation}
Moreover, writing $\omega_j=\theta+\ddc\varphi_j$, we have, with
constants independent of $j$,
\begin{equation}\label{eq:kahleruniform}
 \|\varphi_j\|_\infty\le C,\qquad
 \omega_j^n\le CH\Omega,\qquad
 \omega_j\ge C^{-1}\theta,\qquad
 \operatorname{diam}(\overline X_j,d_j)\le D.
\end{equation}
Here $(\overline X_j,d_j)$ denotes the intrinsic metric completion
associated with $\omega_j$. Each $\omega_j$ satisfies the adapted
density condition with the common exponent $p_0>1$ and a common bound.
\end{lemma}
\begin{proof}
We apply Lemma~\ref{lem:twistregularization} to the normalized
quasi-plurisubharmonic function $q=\varphi+\psi$, whose extension to
$X$ is obtained as in \cite[proof of Proposition 3.1]{CCHSTT}.
We then solve the Monge--Amp\`ere equations with positive exponential
in \cite[equations (2.3)--(2.4)]{FGS} with these smooth twists.
The proofs of \cite[Lemmas 2.5 and 6.1]{FGS} carry over, using
the estimates for bounded potentials \cite{EGZ,DP}, the diameter
estimate of Section~\ref{sec:prelim}, and
\cite[Theorem 1.2]{CCHSTT} for the uniform lower bound
$\omega_j\ge c\theta$.
The uniform integrability hypothesis of the latter theorem is
verified by the density reduction of Lemma~\ref{lem:density}, applied
also with respect to $\theta^n$. The fixed bound gives
\eqref{eq:kahleruniform}, and the lower bound for the metrics converts
$\ric(\omega_j)\ge-\omega_j-C_0\theta$ into
\eqref{eq:kahlerapprox}.
\end{proof}

\subsection{Spectral approximations}
\label{subsec:singularspectral}

\begin{lemma} 
\label{lem:singularspectral}
Let $\omega_j$ be the metrics of
Lemma~\ref{lem:kahlerapprox}, and suppose that their metric
completions, equipped with the Riemannian measures, are
$\RCD(-A,2n)$ spaces. Then every Friedrichs eigenfunction of the
target metric $\omega$ is Lipschitz with respect to its intrinsic
distance, and extends to a Lipschitz function on its metric
completion.
\end{lemma}
\begin{proof}
We use the heat kernel and form domain results in
\cite[Theorem 2.2, Lemma 8.2 and Proposition 8.1]{GPSS2}, together
with the gradient estimate for the heat semigroup on RCD spaces
\cite{AGS}. The min--max argument of Lemma~\ref{lem:spectral} applies
on the common smooth locus, since the domination in
\eqref{eq:kahleruniform} prevents any loss of $L^2$ mass at the
singular set. This argument recovers a complete eigenbasis for the
target, and the uniform gradient bounds pass to the locally smooth
limit.
\end{proof}

\subsection{The Ricci lower bound}

\begin{proof}[Proof of Theorem~\ref{thm:main}]
By Lemma~\ref{lem:kahlerapprox} and
Proposition~\ref{prop:smoothtwist}, the metric completions of the
approximating metrics are noncollapsed $\RCD(-A,2n)$ spaces.
It then follows from Lemma~\ref{lem:singularspectral} that every
eigenfunction of the target is Lipschitz. Honda's criterion, in the
form of Proposition~\ref{thm:criterion} (cf.\ \cite[Lemma 3.3]{FGS}),
now applies with the original bound $\ric(\omega)\ge-\omega$,
and gives the noncollapsed $\RCD(-1,2n)$ conclusion for
$\omega^n$. The assertion for $\mu=\omega^n$ follows by rescaling
the measure by a constant.
\end{proof}

We note that only the regularization argument of
\cite[Section 6]{FGS} is used here, with its resolution
assumption~(5.1) replaced by Proposition~\ref{prop:smoothtwist}.
In particular, no prior topological identification of the metric
completions is needed.

\subsection{Local holomorphic geometry and singular sets}
\label{subsec:local-singularities}

In this subsection, we adapt the local holomorphic arguments of
\cite[Section 4]{Szekelyhidi} and \cite[Sections 5--6]{FGS}.
Let $Z=\overline X$, $M= X_{\reg}$, and $\Gamma=Z\setminus M$.
By Lemma~\ref{lem:kahlerapprox} and
\cite[Theorem 1.2]{CCHSTT}, we have $\omega\ge c\theta$ on $M$.
Therefore the identity map extends to a Lipschitz surjection
$\iota:Z\to X$ with $\iota^{-1}( X_{\reg})=M$.
We choose Stein neighborhoods $U_0\Subset U_1\Subset U_2$ on which
$\omega=\ddc\phi_U$ with $\phi_U$ bounded, and set
$\widehat U_a=\iota^{-1}(U_a)$.
The trivial bundle $L_U=U_2\times\mathbb C$, equipped with the metric
$h_U=e^{-\phi_U}$, provides a local polarization for any
real K\"ahler class; compare \cite[the remark after Lemma 5.15]{FGS}
and \cite[Appendix A]{FGSW}.

\begin{lemma} \label{lem:local-sections}
The weighted $L^2$ correction, as well as the interior estimates for
sections and their gradients used in the construction of Gaussian
sections, hold for $L_U^k$ for all sufficiently large integers $k$.
When the norms are measured with respect to $h_U^k$ and $k\omega$,
the constants on $\widehat U_0$ are independent of $k$.
Moreover, the corrected holomorphic sections extend over $U_1$.
\end{lemma}
\begin{proof}
We use the local versions of \cite[Lemmas 4.1--4.3]{FGS}, together
with H\"ormander's estimate as in \cite[Theorem 18]{Szekelyhidi}.
The required positivity is given by
$k\omega+\ric(\omega)\ge(k-1)\omega$.
The cutoff functions of \cite[Lemma 8.2 and Proposition 8.1]{GPSS2}
justify the Bochner and Moser arguments across $\Gamma$.
Finally, the lower bound for the metric, together with normality,
gives the extension of the corrected holomorphic sections, as in
\cite[Lemma 5.9]{FGS} and \cite[Appendix A]{FGSW}.
\end{proof}

We use the notation
\begin{equation}\label{eq:local-volume-density}
 \nu_Z(z)=\lim_{r\downarrow0}
 \frac{\mathcal H^{2n}(B(z,r))}{\alpha_{2n}r^{2n}},
 \qquad \mathcal R_\epsilon(Z)=\{\nu_Z>1-\epsilon\},
\end{equation}
where $\alpha_{2n}$ denotes the volume of the Euclidean unit ball.

\begin{lemma} \label{lem:local-density-gap}
There exists $\epsilon_X>0$ such that
\[
 \nu_Z(z)\le1-\epsilon_X\quad(z\in\Gamma),
 \qquad \mathcal R_{\epsilon_X}(Z)=\mathcal R(Z)=M.
\]
In particular, the metric singular set of $Z$ is closed.
\end{lemma}
\begin{proof}
We localize the finite order and three-annulus arguments of
\cite[Lemmas 20--22 and Propositions 23--24]{Szekelyhidi}.
The klt integrability condition and the $L^p$ bound for the density
provide the finite order estimate for local holomorphic
functions vanishing on $X_{\mathrm{sing}}$.
In the proof of Proposition~23 there, we replace the convergence step
for K\"ahler--Einstein metrics by the holomorphic charts of
\cite[Theorem 1.4 and Propositions 2.4, 3.2]{LiuSzekelyhidi} under
Ricci lower bounds, using Lemma~\ref{lem:local-sections} for the
weighted corrections. A finite covering of $X$ then gives the stated
density gap on the fixed space.
\end{proof}

\begin{lemma}[Local cones and cutoffs]\label{lem:local-cones}
The following statements hold for the fixed space $Z$.
\begin{enumerate}[label=(\roman*)]
\item No blow-up sequence centered in $\Gamma$ can converge
to a cone $\mathbb R^{2n-2}\times C(S^1_\gamma)$ with
$0<\gamma<2\pi$.
\item For some $\epsilon>0$, the set
$V\setminus\mathcal R_\epsilon(V)$ has locally zero
$W^{1,2}$ capacity for every tangent cone $V$ of $Z$.
\item Gaussian sections and point-separating sections can
be constructed in the local bundles $L_U^k$ at every
point of $\widehat U_0$.
\end{enumerate}
\end{lemma}
\begin{proof}
For (i), we apply the cone exclusion arguments of
\cite[Proposition 10]{Szekelyhidi} and
\cite[Lemmas 6.4--6.5]{FGS} locally, using
Lemma~\ref{lem:local-density-gap} and the holomorphic charts
of \cite{LiuSzekelyhidi} under Ricci lower bounds.
Part (ii) is a localization of \cite[Lemma 6.7]{FGS},
using the cutoff construction of
\cite[Proposition 5.1]{LiuSzekelyhidi}.
Part (iii) follows from the construction of Gaussian sections in
\cite[Proposition 3.1]{LiuSzekelyhidi} and
\cite[Lemma 6.4 and Corollary 6.1]{FGS}, where
Lemma~\ref{lem:local-sections} provides the corrections in
$L_U^k$; the local formulation is also described in
\cite[Appendix A]{FGSW}.
\end{proof}

\begin{proof}[Proof of Corollary~\ref{cor:main}]
Applying the point separation argument of
\cite[Lemma 5.9 and Corollary 6.2]{FGS} to the local
sections of Lemma~\ref{lem:local-cones}(iii), we see that $\iota$
is injective, and hence a homeomorphism.
Lemma~\ref{lem:local-density-gap} identifies the regular locus
and shows that the singular set is closed.
The cone exclusion in Lemma~\ref{lem:local-cones}(i), combined
with the stratification of RCD spaces
\cite[Theorem 1.8]{DePhilippisGigli}, gives
\eqref{eq:cor-codim-three}, as in \cite[Lemma 6.5]{FGS}.
Under the additional Ricci upper bound, we apply the argument of
\cite[Lemma 6.6]{FGS} for Ricci-flat tangent cones,
using \cite{CCT} and \cite[Propositions 27--28]{Szekelyhidi}.
This excludes nonflat tangent cones splitting off a factor
$\mathbb R^{2n-3}$, and the stratification then gives
\eqref{eq:cor-codim-four}.
\end{proof}

\section{The polarized partial \texorpdfstring{$C^0$}{C0} estimate}
\label{sec:partialc0}

The proof follows \cite[Section 9]{FGS}, now in arbitrary dimension,
where Theorem~\ref{thm:main} and Corollary~\ref{cor:main} provide
the RCD structure and the identification of the regular locus.
We write $dV_\omega=\omega^n$ and denote by
$\mathscr P(n,D,v)$ the polarized class considered in
Theorem~\ref{thm:partialc0}.

\subsection{Uniform analytic estimates}
By RCD comparison geometry, the noncollapsing and Sobolev constants
are uniform and depend only on $n,D,v$. Moreover, the Sobolev
inequality remains uniform under the rescaling $\omega\mapsto k\omega$,
$k\ge1$; see also \cite{ZhangPartial}.
Consequently, the proofs of \cite[Lemmas 4.1--4.3]{FGS}, together
with the cutoff functions of zero capacity from
\cite[Lemma 8.2]{GPSS2}, give
\begin{equation}\label{eq:polarized-sections}
 \|s\|_{L^\infty(h^\ell)}
 +\|\nabla s\|_{L^\infty(h^\ell,\ell\omega)}
 \le C(n,D,v)\|s\|_\ell
 \quad(s\in H^0(X,L^\ell),\ \ell\ge1).
\end{equation}
The $L^2$ correction estimate is uniform for $\ell\ge2$, since
$\ell\omega+\ric(\omega)\ge\tfrac12\ell\omega$;
compare \cite[Theorem 18]{Szekelyhidi}.
The local lower bound for the metric needed for the extension of
sections is provided by \cite[Theorem 1.2]{CCHSTT}.

\subsection{The  peak section constructions}
We use the normalized volume density defined in
\eqref{eq:local-volume-density}. The following lemma is the uniform
version of the geometric input.

\begin{lemma}\label{lem:polarized-cones}
For the class $\mathscr P(n,D,v)$, the following statements hold.
\begin{enumerate}[label=(\roman*)]
\item There exists $\epsilon_n>0$ such that
$\nu_X(x)\le1-\epsilon_n$ at every $x\in X\setminus  X_{\reg}$.
\item Let $Z$ be a pointed noncollapsed limit of spaces
$(X_j,\sqrt{k_j}d_j,x_j)$, where
$(X_j,\omega_j,L_j,h_j)\in\mathscr P(n,D,v)$ and $k_j\ge1$.
Every tangent cone $V$ of $Z$, and every iterated tangent cone
of $V$, has an open regular locus carrying a smooth Ricci-flat
K\"ahler metric. Its closed singular set has locally zero
$W^{1,2}$ capacity. Compact subsets of that regular locus
are obtained, along a diagonal rescaling sequence, from
subsets of $X_j^\circ$ with convergence of the polarized
K\"ahler data sufficient for the peak section construction.
\end{enumerate}
\end{lemma}
\begin{proof}
We adapt the induction on dimension in
\cite[proof of Theorem 36]{Szekelyhidi}, starting from
Lemma~\ref{lem:local-density-gap} and the cone exclusions in
\cite[Lemmas 5.14--5.16 and 6.6]{FGS}.
The two-sided Ricci bound ensures that the regular parts are smooth
and Ricci-flat after blow-up, while \cite{CCT} excludes nonflat
cones of codimension two. This gives the density gap in (i), which
depends only on the dimension.
The limiting cone argument of \cite[Section 9]{FGS}, together
with the volume convergence for noncollapsed spaces
\cite{DePhilippisGigli}, then gives (ii): the uniform gap keeps
analytic singularities away from the regular region, and the
stratification provides the cutoff functions of zero capacity.
These are precisely the ingredients needed for the polarized
convergence and transplantation arguments in
\cite[Section 3]{DonaldsonSun}.
\end{proof}

\subsection{A common power and all its multiples}
\begin{proof}[Proof of Theorem~\ref{thm:partialc0}]
We apply the peak section argument of \cite[Section 3]{DonaldsonSun}
in the singular form developed in \cite[Section 9]{FGS}, using
Lemma~\ref{lem:polarized-cones} and
\eqref{eq:polarized-sections}. The uniform Sobolev inequality for the
rescaled metrics allows us to apply the argument of
\cite{ZhangPartial} for all multiples:
there exist $Q\ge2$ and $\eta>0$, depending only on $n,D,v$, such
that for every $x$ and $k\ge1$, some $2\le q\le Q$ satisfies
$\rho_{qk,\omega}(x)\ge\eta$.
Taking tensor powers up to the common degree $Q!$ gives the
uniform lower bound for all of its multiples, while the upper bound
follows from \eqref{eq:polarized-sections}.
The construction of sections separating points and coordinates in
\cite[Section 4]{DonaldsonSun}, as used in
\cite[Section 9]{FGS}, gives a uniformly very ample power, and
taking a common multiple yields the required $m$.
Finally, integrating the upper bound for the Bergman kernel and
using the upper bound for the volume, we obtain the uniform ambient
dimension $N$.
\end{proof}

\section{Proofs of the fundamental-group and stability applications}
\label{sec:applications-proofs}

\subsection{Geometric stability and canonical distances}
\label{subsec:stability-proof}

Throughout this subsection, we use the uniform H\"older estimate from \cite{GSSHolder}.
The arguments below are adaptations to klt spaces of those in
\cite[Sections 3--6]{GSSStability}.

\begin{proof}[Proof of Theorem~\ref{thm:canonical-distance}]
We follow \cite[Lemmas 6.1--6.3]{GSSStability}.
The smooth approximation theorem
\cite[Corollary 2.5]{EGZApprox} gives
$f_j\in C^\infty(X)\cap\PSH(X,\lambda\theta)$ with
$f_j\searrow f_\omega$.
Using \cite{EGZ,DP}, we solve $\omega_j^n=a_je^{-f_j}\Omega$, where
the constants $a_j$ normalize the mass.
By dominated convergence, the densities satisfy common $L^p$ bounds
and converge in total variation.
The estimate \cite[Theorem 1.2]{CCHSTT} gives uniform lower bounds
for the metrics and their Ricci curvature, so that
Theorem~\ref{thm:main} applies.
By Theorem~\ref{thm:main}, Corollary~\ref{cor:main} and the
almost everywhere convexity of the regular set \cite{Deng}, the
comparison and segment arguments of
\cite[Lemmas 3.1, 4.1--4.2 and 5.1--5.4]{GSSStability} extend to
$ X_{\reg}$. Together with the uniform H\"older distance estimates and the
Bishop--Gromov lower volume bounds, they imply that the distances
are uniformly Cauchy. The uniform Ho\"der distance estimate 
identifies the topology and separates points.
The independence of the regularization follows by interleaving two
admissible sequences, as in \cite[Lemma 6.2]{GSSStability}, and the
comparison with a constant sequence gives compatibility with the
intrinsic distance whenever the latter is defined.
\end{proof}

\subsection{Almost-Abelian fundamental groups in every dimension}
\label{subsec:fundamental-proof}

\begin{proof}[Proof of Theorem~\ref{thm:fundamental}]
Since the case of curves is classical, we may assume $n\ge2$.
We follow the proof of \cite[Theorem 1.1 and Corollary 1.1]{FGSW}.
For $\Delta=0$, Theorem~\ref{thm:main} and
Corollary~\ref{cor:main} replace the three-dimensional RCD
input \cite[Proposition 2.1]{FGSW} in arbitrary dimension.
For klt boundaries, we use the regularization in
\cite[Section 2.2.2, Proposition 2.3]{FGSW}; for each fixed
$\eps>0$, the geometric stability argument in
Section~\ref{subsec:stability-proof} identifies the limit on $X$, and
Honda's criterion \cite{Honda} recovers the sharper
$\RCD(-\eps,2n)$ bound, as in that proposition.
The boundary perturbation of \cite[Lemma 2.10]{FGSW} then
gives $\RCD(-2\eps,2n)$ metrics in a fixed K\"ahler class
for the log canonical pair.

The uniform estimate for the lengths of generators and the RCD
Margulis argument \cite[Theorem 2.2 and Lemmas 2.11--2.12]{FGSW}
show that $\pi_1(X)$ is virtually nilpotent; we note that no diameter
bound uniform in $\eps$ is needed.
The surjectivity of the Albanese map on finite \emph{\'etale} covers,
together with the group theoretic argument in
\cite[Theorem 1.3, Section 3.5 and Appendix C]{FGSW},
then gives almost-Abelianity. None of these arguments depends on
the complex dimension.
\end{proof}

\begin{proof}[Proof of Corollary~\ref{cor:calabi-yau}]
Since $c_1(K_X+\Delta)=0$ implies that $-(K_X+\Delta)$ is nef, the
corollary follows from Theorem~\ref{thm:fundamental}; the
Calabi--Yau case corresponds to $\Delta=0$.
\end{proof}

\end{document}